\documentclass[10pt,twoside, a4paper, english, reqno]{amsart}
\usepackage{graphicx, amsmath, varioref, amscd, amssymb,color, bm, stmaryrd, amsthm}
\usepackage{fancybox}
\usepackage[pdftex, colorlinks=true]{hyperref}
\usepackage{pdfsync}

\numberwithin{equation}{section}
\allowdisplaybreaks

\newtheorem{theorem}{Theorem}[section]
\newtheorem{lemma}[theorem]{Lemma}
\newtheorem{corollary}[theorem]{Corollary}

\newtheorem{proposition}[theorem]{Proposition}
\theoremstyle{definition}

\theoremstyle{remark}
\newtheorem{remark}[theorem]{Remark}

\newcommand{\Div}{\operatorname{div}}

\newcommand{\Grad}{\nabla}
\newcommand{\vr}{\varrho}

\def\pa{\partial}
\def\na{\nabla}
\def\N{\mathbb N}
\def\vphi{\varphi}

\newcommand{\weak}{\rightharpoonup}

\newcommand{\weakstar}{\overset{\star}\rightharpoonup}

\newcommand{\Set}[1]{\left\{#1\right\}}

\newcommand{\R}{\mathbb{R}}

\newcommand{\F}{\mathcal{F}}
\newcommand{\E}{\mathcal{E}}
\newcommand{\conv}{\overset{n \rightarrow \infty}{\longrightarrow}}

\newcommand{\eps}{\epsilon}

\begin{document}

\title[Existence of solutions to Kinetic flocking models]
	  {Existence of weak solutions to \\kinetic flocking models}

\author[Karper]{Trygve K. Karper}\thanks{The work of T.K was supported by the Research Council of Norway through the project 205738}

\address[Karper]{\newline
Center for Scientific Computation and Mathematical Modeling, University of Maryland, College Park, MD 20742}
\email[]{\href{karper@gmail.com}{karper@gmail.com}}
\urladdr{\href{http://folk.uio.no/~trygvekk}{folk.uio.no/\~{}trygvekk}}

\author[Mellet]{Antoine Mellet}\thanks{The work of A.M was supported by the National Science Foundation  under the Grant  DMS-0901340}

\address[Mellet]{\newline
Department of Mathematics, University of Maryland, College Park, MD 20742}
\email[]{\href{mellet@math.umd.edu}{mellet@math.umd.edu}}
\urladdr{\href{http://www.math.umd.edu/~mellet}{math.umd.edu/~{}mellet}}

\author[Trivisa]{Konstantina Trivisa}\thanks{The work of K.T. was supported by the National Science Foundation   under the Grant  DMS-1109397}
\address[Trivisa]{\newline
Department of Mathematics, University of Maryland, College Park, MD 20742}
\email[]{\href{trivisa@math.umd.edu}{trivisa@math.umd.edu}}
\urladdr{\href{http://www.math.umd.edu/~trivisa}{math.umd.edu/~{}trivisa}}

\date{\today}




\maketitle

\begin{abstract}
We establish the global existence of weak solutions 
to a class of kinetic flocking equations.
The  models under conideration  
include the kinetic Cucker-Smale equation \cite{CS-2007A, CS-2007B}
with possibly non-symmetric flocking potential, 
the Cucker-Smale equation with additional strong local alignment, 
and a newly proposed model by Motsch and Tadmor \cite{MT-2011}. 
The main tools employed in the analysis 
are the velocity averaging lemma and 
the Schauder fixed point theorem along with 
various integral bounds. 

\end{abstract}

\tableofcontents{}

\section{Introduction and main results}
Models describing collective self-organization 
of biological agents are currently receiving  considerable attention. 
In this paper, we will study a class of such models. 
More precisely, we focus on kinetic type models for the flocking behavior exhibited by certain species of birds, fish, and insects: 
Such models are typically of the form
\begin{equation}\label{eq:eq1}
	f_t + v\cdot \Grad_x f  + \Div_v\left(fL[f]\right)   + \beta \Div_v(f(u -v)) = 0,\quad\mbox{ in } \R^d\times\R^d\times(0,T)
\end{equation}
where $f:= f(t,x,v)$ is the scalar unknown,  $d\geq 1$ is the spatial dimension, and $\beta \geq 0$
is a constant.
The first two terms describe the free transport of the individuals, and the last two terms 
take into account the interactions between individuals, who try to align with their neighbors.
The alignment operator $L$ has the form
\begin{equation}\label{eq:CSop0}
L[f] = \int_{\R^d} \int_{R^d} K_f(x,y)  f(y,w) (w-v)\, dw\, dy,
\end{equation}
where the kernel $K_f$ may depend on $f$ and may not be symmetric in $x$ and $y$ (see \eqref{eq:MT} below).
The last term in (\ref{eq:eq1}) describes  strong local alignment interactions (see below), where  $u$ denotes the average local velocity,  defined  by 
$$u(t,x) = \frac{\int_{\R^d}fv~dv}{\int_{\R^d}f~dv}.$$

Equation \eqref{eq:eq1} includes the classical kinetic Cucker-Smale model,  which corresponds to
$\beta=0$ and an alignment operator $L[f]$ given by (\ref{eq:CSop0}) with 
 a smooth kernel independent of $f$: 
$$K_f(x,y)=K_0(x,y)$$
with $K_0$ symmetric ($K_0(x,y)=K_0(y,x)$).
There is a considerable body of literature concerning the kinetic Cucker-Smale 
equation and its variations (see \cite{CR-2010, CFR-2010, CFT-2010, CCR-2011,CS-2007A, CS-2007B,DFT-2010, HL-2009, HT-2008, MT-2011}), but a general existence 
theory  has thus far remained absent.
Notable exceptions are the studies \cite{BCC-2011,CCR-2011} in which 
well-posedness for Cucker-Smale-type models is established in the sense of measures and then extended 
to weak solutions (\cite{BCC-2011} adds noise to the model). 
The existence of classical solutions to the kinetic Cucker-Smale equation is 
established in \cite{HT-2008}, but the result does not extend to \eqref{eq:eq1}. 
Compared to these results, the main contribution of this paper 
is the addition of the local alignment term ($\beta > 0$) and 
the fact that the function $K_f$ is allowed to be non-symmetric (such models may not preserve the total momentum). 
More precisely, our analysis will include  a  model recently proposed by Motsch \& Tadmor \cite{MT-2011}, 
for which  $K_f$ depends on $f$ as follows:
\begin{equation}\label{eq:Kf}
K_f(x,y) =\frac{ \phi(x-y)}{\phi \star \int f~dv}
\end{equation}
with $\star$ denoting the convolution product in space.

In the remaining parts of this introduction, we first 
introduce the models we consider in more details and 
discuss various variations proposed in the literature. 
Following this we introduce the notion of weak solutions.
We end the introduction by stating our main existence result 
and providing a brief sketch of the main ideas used to 
prove it.

\subsection{The kinetic Flocking models}
Our starting point
is the pioneering model introduced by Cucker and Smale \cite{CS-2007A}: Consider 
$N$ individuals (birds, fish) each totally described by 
a position $x_i(t)$ and a velocity $v_i(t)$. The Cucker-Smale model \cite{CS-2007A, CS-2007B}
is given by the evolution
\begin{equation}\label{eq:particle}
	\dot x_i = v_i, \qquad \dot v_i = -\frac{1}{N}\sum_{j\neq i} K_0(x_i,x_j)(v_j - v_i).
\end{equation}
Roughly speaking, each particle attempts
to align it's velocity with a local average velocity 
given by the form and support of $K_0$.
In this paper, we focus not  on the particle model but  on the corresponding 
kinetic description. This description can be directly derived 
for the 
empirical distribution function
$$
f(t,x,v) = \frac{1}{N}\sum_i \delta(x-x_i(t))\delta(v-v_i(t)).
$$
Specifically, by direct calculation 
one sees that $f$ evolves according to
\begin{equation}\label{eq:eq-1}
	f_t + v\cdot \Grad_x f + \Div_v\left(fL[f]\right) = 0,
\end{equation}
where $L$ is given by
\begin{equation}\label{eq:CSop}
L[f] = \int_{\R^d} \int_{R^d} K_0(x,y)  f(y,w) (w-v)\, dw\, dy,
\end{equation}

The question we will address in this paper is 
whether or not $f$ is a function 
when the initial data $f_0$ is a function. For this purpose,
it is desirable that  we do not have loss of  mass at infinity.
Note that there is no effect countering such loss in  \eqref{eq:eq-1}.  
Indeed, \eqref{eq:eq-1} consists of only transport and alignment 
of the velocity to an average velocity. 
Hence, if the average velocity 
is non-zero, the equation will eventually transport all mass to infinity.  
To counter this, we add a confinement potential $\Phi$ to the equation.
The only property we require of this potential is that $\Phi \rightarrow \infty$ 
when $|x| \rightarrow \infty$. 

\vspace{0.3cm}

\subsubsection{The Cucker-Smale:}
The simplest model that we consider in this paper, is thus the Cucker-Smale model \cite{CS-2007A, CS-2007B} with confinement potential:

{\bf Cucker-Smale:}
\begin{equation}\label{eq:eq0}
	f_t + v\cdot \Grad_x f - \Div_v(f \Grad_x \Phi)+ \Div_v\left(fL[f]\right) = 0.
\end{equation}
where $L$ is given by (\ref{eq:CSop}) with $K_0$ a smooth symmetric function.
\vspace{0.3cm}

\subsubsection{The Motsch-Tadmor correction}
In the recent paper by Motsch and Tadmor \cite{MT-2011}
it is argued that the normalization factor $\frac{1}{N}$
in \eqref{eq:particle} leads to some undesirable features.  
In particular, if a small group of individuals 
are located far away from a much 
larger group of individuals, 
the internal dynamics in the small group is almost 
halted since the number of individuals is large.  
This is easily understood if we also assume that $K_0$ 
has compact support and that the distance between 
the two groups are larger than the support of $K_0$.
In that case, the dynamics of the two groups should be independent
of each other, but the size of the alignment term in Cucker-Smale model
still depends on the total number of individuals. 
To remedy this, Motsch and Tadmor \cite{MT-2011} propose a new model, with normalized, non-symmetric alignment, given by
\begin{equation}\label{eq:MT} 
\tilde L[f] = \frac{\int_{\R^d} \int_{R^d} \phi(x-y)  f(y,w) (w-v)\, dw\, dy}{\int_{\R^d} \int_{R^d} \phi(x-y)  f(y,w)\, dw\, dy}
\end{equation}
(which is the general operator (\ref{eq:CSop0}), with $K_f$ given by (\ref{eq:Kf})).

The resulting model is the following:

{\bf Motsch-Tadmor:}
\begin{equation}\label{eq:eq1'}
	f_t + v\cdot \Grad_x f - \Div_v(f\Grad_x \Phi) + \Div_v\left(f\tilde L[f]\right)
  =0,
\end{equation}
with $\tilde L$ given by (\ref{eq:MT}).
\vspace{10pt}

Note that unlike the Cucker-Smale model \cite{CS-2007A, CS-2007B}, the Motsch-Tadmor model \cite{MT-2011} does not preserve the total momentum $\int \int vf\, dv\, dx$.

\subsubsection{Local alignment}
It is also possible to combine the 
Cucker-Smale model with the Motsch-Tadmor 
model letting the Cucker-Smale flocking 
term dominate the long-range interaction and
the Motsch-Tadmor term dominate short-range interactions. 
This will correct the aforementioned 
deficiency of the kinetic Cucker-Smale model.
However, the large-range interactions is still close to that of the Cucker-Smale model. 
In particular, we consider the singular limit 
where the Motsch-Tadmor flocking kernel $\phi$ converges to a Dirac distribution. The Motsch-Tadmor correction then
converges to a local alignment term given by:
$$ \tilde L[f] =\frac{ j-\rho v }{\rho} = u-v.
$$
where 
$$ \rho(x,t) = \int_{\R^d} f(x,v,t)\, dv , \qquad j(x,t)=\int_{\R^d} v f(x,v,t)\, dv $$
and $u(x,t)$ is defined by the relation
\begin{equation}\label{eq:uu} 
u = \frac{\int_{\R^d} v f\, dv}{\int_{\R^d} f\, dv}.
\end{equation}
This leads to the following equation:

{\bf Cucker-Smale with strong local alignment:}
\begin{equation}\label{eq:eq}
	f_t + v\cdot \Grad_x f - \Div_v(f\Grad_x \Phi) + \Div_v\left(fL[f]\right)   + \beta  \Div_v(f(u -v)) = 0 
\end{equation}
with $L$ given by (\ref{eq:CSop}) with $K$ a given symmetric function.

Note that though (\ref{eq:eq}) is obtained as singular limit of the non-symmetric Motsch-Tadmor model, it has more symmetry, and in particular preserves the total momentum (we will see later that it also has good entropy inequality).

\subsubsection{Noise, self-propulsion, and friction}
Finally, for the purpose of applications, there are many other aspects 
that are not included in the models we have considered 
so far. For instance, there might be unknown 
forces acting on the individuals such as wind 
or water currents. Often these type of effects 
are simply modeled as noise. If this noise 
is brownian it will lead to the addition of 
a Laplace term in the equations. 
We note that this term has a regularizing effect on the solutions, 
but that this regularizing effect is not required to prove the existence of solutions.
In addition, we could 
add self-propulsion and friction in the models. 
This amounts to adding a term $ -\Div((a-b|v|^2)vf)$ in the equation.
The most general model for which we will able to
 prove  global existence of solutions is the following:

{\bf Cucker-Smale with strong local alignment, noise, self-propulsion, and friction:}
\begin{equation}\label{eq:eqn}
	\begin{split}
		&f_t + v\cdot \Grad_x f - \Div_v(f\Grad_x \Phi) + \Div_v\left(fL[f]\right)   +   \beta \Div_v(f(u -v)) \\
		&= \sigma \Delta_v f  -\Div((a-b|v|^2)vf)
	\end{split}
\end{equation}
with $\sigma\geq 0$, $a\geq 0$ and $b\geq 0$.


%
%
	

\subsection{Main results}
We now list our main results.
The existence of solutions for (\ref{eq:eq0}) when the alignment operator is given by (\ref{eq:CSop}) with smooth bounded symmetric kernel $K_0$ presents no particular difficulties.
Our focus, instead, is on the local alignment models (\ref{eq:eq}) and (\ref{eq:eqn}) (with or without noise, friction and self-propulsion).

Throughout this paper, the potential $\Phi(x)$ is a smooth confinement potential, satisfying
$$ \lim_{|x|\to \infty} \Phi(x)=+\infty.$$
\subsubsection{Existence with local alignment and $K$ symmetric} Our first result is:
\begin{theorem}\label{thm:main}
Assume that $f_0\geq 0$ satisfies 
$$
f_0\in L^\infty(\R^{2d})\cap L^1(\R^{2d}), \quad \mbox{and} \quad  ( |v|^2+\Phi(x)) f_0 \in  L^1(\R^{2d}).
$$
Assume that $L$ is the alignment operator  given by \eqref{eq:CSop}
with $K_0$ symmetric ($K_0(x,y)=K_0(y,x)$) and bounded.
Then, for any $\sigma\geq 0$, $a\geq 0$ and $b\geq 0$  there exists 
$f$ such that
$$ f \in C(0,T;L^1(\R^{2d}))\cap L^\infty((0,T)\times\R^{2d}), \qquad ( |v|^2+\Phi(x)) f \in L^\infty (0,\infty;L^1(\R^{2d})), $$
and $f$ is a solution of (\ref{eq:eqn}) in the following weak sense: 
\begin{equation}\label{eq:weak}
		\begin{split}
				&\int_{\R^{2d+1}} - f \psi_t - vf\Grad_x \psi +f\Grad_x \Phi \Grad_x \psi - fL[f]\Grad_v \psi ~dvdxdt\\
				&+ \int_{\R^{2d+1}}\sigma \Grad_v f\Grad_v \psi -\beta f(u-v)\Grad_v \psi~dvdxdt \\
				& + \int_{\R^{2d+1}} (a-b|v|^2)vf  \na_v \psi ~dvdxdt
				= \int_{\R^{2d}} f^0 \psi(0,\cdot)~dvdx,
		\end{split}
	\end{equation}
for any $\psi \in C_c^\infty([0,T)\times \R^{2d})$, where $u$ is such that $j=\rho u$.
\end{theorem}

\begin{remark}
\item Note that the definition of $u$ is ambiguous if $\rho$ vanishes (vacuum). We thus define $u$ pointwise by
\begin{equation}\label{eq:uu1} 
u(x,t)=\left\{ \begin{array}{ll}\displaystyle \frac{j(x,t)}{\rho(x,t)} & \mbox{ if }  \rho(x,t) \neq 0 \\[8pt] 0 &  \mbox{ if }  \rho(x,t) = 0 \end{array}\right.
\end{equation}
Since 
$$j \leq  \left(\int |v|^2 f(x,v,t)\,  dv\right)^{1/2} \rho^{1/2}$$
we have $j=0$ whenever $\rho=0$ and so (\ref{eq:uu1}) implies $j=\rho u$.

\item We note also that $u$ does not belong to any $L^p$ space. However, we have
$$ \int_{\R^{2d}} |uf|^2\, dx\, dv \leq \|f\|_{L^\infty(\R^{2d})}  \int_{\R^{2d}} |v|^2 f(x,v,t)\,  dv\, dx$$
so that the term $u f$ in the weak formulation (\ref{eq:weak})  makes sense as a function in $L^2$.
\end{remark}

The proof of  Theorem \ref{thm:main} is developed in Sections \ref{sec:approximate} and \ref{sec:existence}.
The main difficulty is  of course the nonlinear term $fu$ (the other nonlinear term $fL[f]$ is much more regular, since $L[f]\in L^\infty$).
Because $u$ does not belong to any $L^p$ space, we cannot do a fixed point argument directly to prove the existence. Instead, we will introduce an approximated equation (see \eqref{eq:app}), in which $u$ is replaced by a more regular quantity.
Existence for this approximated equation will be proved via a classical Schauder fixed point argument (compactness will follow from velocity averaging lemma and energy estimates).
The main difficulty is then to pass to the limit in the regularization, which amounts to proving some stability property for (\ref{eq:eqn}). 
As pointed out above, since $u$ cannot be expected to converge in any $L^p$ space, we will pass to the limit in the whole term $fu$ and then show that the limit has the desired form.
We note that the  friction and self-propulsion term introduces no additional difficulty. We will thus take $a=b=0$ throughout the proof. 
The noise term has a regularizing effect, which will not be used in the proof. We thus assume that $\sigma\geq 0$.

\vspace{10pt}

\subsubsection{Existence with local alignment and  the  Motsch-Tadmor term}
Our second result concerns  the  Motsch-Tadmor model \eqref{eq:eq1'} with $\tilde L$ given by \eqref{eq:MT}. We can rewrite (\ref{eq:eq1'}) as
\begin{equation}\label{eq:eqMT}
f_t + v\cdot \Grad_x f - \Div_v(f\Grad_x \Phi)   +   \Div_v(f(\tilde u -v)) = 0 
\end{equation}
where
$$ 
\tilde u(x,t) = \frac{\int_{\R^d} \int_{R^d} \phi(x-y)  f(y,w,t) w \, dw\, dy}{\int_{\R^d} \int_{R^d} \phi(x-y)  f(y,w,t)\, dw\, dy}= \frac{\int_{R^d} \phi(x-y) j(y,t)\, dy}{\int_{R^d} \phi(x-y) \rho(y,t)\, dy}.
$$
Compared with (\ref{eq:eq}), the term $f(\tilde u -v)$ is thus less singular than $f(u-v)$ (which we recover when $\phi(x-y)=\delta(x-y)$), and so the existence of solution for (\ref{eq:eq1'})
can be proved following similar (or simpler) arguments, provided the necessary energy estimate hold. And this turns out to be quite delicate. 
Indeed, the lack of symmetry of the alignment operator $\tilde L$ implies that it does not preserve momentum, and proving that the energy $\int (|v|^2 +\Phi(x)) f(x,v,t)\, dv\, dx$ remains bounded for all time proves delicate for general $\phi$.
We will prove an existence result when $\phi$ is compactly supported.
More precisely, we have:
\begin{theorem}\label{thm:MT}
Assume that $f_0\geq 0$ satisfies 
$$
f_0\in L^\infty(\R^{2d})\cap L^1(\R^{2d}), \quad \mbox{and} \quad  ( |v|^2+\Phi(x)) f_0 \in  L^1(\R^{2d}).
$$
Assume that $\tilde L$ is the alignment operator (\ref{eq:MT}) where $\phi$ is a smooth nonnegative function such that there exists $r>0$ and $R>0$ such that 
\begin{equation}\label{eq:condphi} 
\phi(x) >0 \mbox{ for } |x|\leq r\, , \qquad  \phi(x)=0  \quad \mbox{ for } |x|\geq R.
\end{equation}
Then there exists a weak solution of (\ref{eq:eqMT}) in the same sense as in Theorem \ref{thm:main}.
\end{theorem}


\vspace{10pt}
\subsubsection{Entropy of flocking with symmetric kernel}
To complete our study, restricting our attention to {\em symmetric} flocking,  we will show that the model \eqref{eq:eq}  is endowed with a {\em natural} dissipative structure.
More precisely, we consider the usual entropy
\begin{equation}
	\mathcal{F}(f) = \int_{\R^{2d}} \frac{\sigma}{\beta} f\log f + f\frac{v^2}{2} + f\Phi~\, dv\,dx
\end{equation}
and the associated dissipations
\begin{equation}
D_1(f)=\frac{1}{2}\int_0^T\int_{\R^{2d}} \frac{\beta}{f}\left|\frac{\sigma}{\beta} \Grad_v f - f(u-v)\right|^2~\, dv\, dx
\end{equation}
and 
\begin{equation}
D_2(f)= \frac{1}{2}\int_{\R^d}\int_{\R^d}\int_{\R^d}\int_{\R^d}K_0(x,y)f(x,v)f(y,w)\left|v - w\right|^2~dwdydvdx. 
\end{equation}
(Note that $D_1(f)$ is a local dissipation due to the noise and local alignment term
while $D_2(f)$ is the dissipation due to the non-local alignment term)

We then have:
\begin{proposition}\label{prop:main}
Assume that $L$ is the alignment operator  given by \eqref{eq:CSop}
with $K_0$ symmetric ($K_0(x,y)=K_0(y,x)$) and bounded, and let  $\beta > 0$ and  $\sigma \ge 0$. 
If $f$ is a  solution of (\ref{eq:eq}) with sufficient integrability, then the following inequality holds:
\begin{align}\label{eq:entropy1}
		&\partial_t \mathcal{F}(f) + D_1(f) + D_2(f) 
		\nonumber\\
		&  \leq \frac{\sigma}{\beta} d \int_{\R^d}\int_{\R^d}\int_{\R^d}\int_{\R^d}K_0(x,y)f(x,v)f(y,w) ~ dwdydvdx. 
\end{align}
Furthermore, if the confinement potential  $\Phi$ satisfies
\begin{equation}\label{eq:conff}
\int_{\R^d} e^{-\Phi(x)}\, dx <+\infty 
\end{equation}
then there exists $C$ depending only on $\|K_0\|_{\infty}$, $\Phi$ and $\int f_0(x,v)\, dx\, dv$ such that
\begin{equation}\label{eq:entropy1.2}
	\begin{split}
		&\partial_t \mathcal{F}(f) + \frac{1}{2}D_1(f) \\ 
		& + \frac{1}{2} \int_{\R^d} \int_{\R^d} K_0(x,y)\vr(x)\vr(y)\left|u(x) - u(y) \right|^2~dydx  
		 \leq \frac{C}{\beta}\mathcal{F}(f(t)).
	\end{split}
\end{equation}
\end{proposition}
The first inequality (\ref{eq:entropy1}) shows that the nonlocal alignment term is responsible for some creation of entropy. The second inequality (\ref{eq:entropy1.2}) shows that this term can be controlled by $D_1(f)$ and the entropy itself. This last inequality is particularly useful  in the study of singular limits of (\ref{eq:eq}) with dominant local alignment ($\beta\rightarrow \infty$).
Such limits will be investigated in \cite{KMT-2012}.

Finally, we can now prove the existence of weak solutions satisfying the entropy inequality:
\begin{theorem}\label{thm:entropy}
Assume that $L$ is the alignment operator  given by \eqref{eq:CSop}
with $K$ symmetric ($K_0(x,y)=K(y,x)$) and bounded, and let  $\beta > 0$ and  $\sigma \ge 0.$ 
Assume furthermore that $f_0$ satisfies
$$
f_0\in L^\infty(\R^{2d})\cap L^1(\R^{2d}), \quad \mbox{and} \quad  ( |v|^2+\Phi(x)) f_0 \in  L^1(\R^{2d}).
$$
then there exist a weak solution of (\ref{eq:eq}) (in the sense of Theorem \ref{thm:main}) satisfying 
\begin{equation}\label{eq:entropyfinal}
\mathcal F(f(t)) + \int_0^t  D_1(f)+ D_2(f) \, ds \leq e^{\frac{\sigma d}{\beta}\|K\|_\infty M^2 t} \mathcal F(f_0)
\end{equation}
and, if $\Phi$ satisfies (\ref{eq:conff}), 
\begin{eqnarray*}
&&\mathcal F(f(t)) + \frac{1}{2} \int_0^tD_1(f)\,ds  \\
&& \qquad\qquad + \frac{1}{2} \int_0^t\int_{\R^d} \int_{\R^d} K_0(x,y)\vr(x)\vr(y)\left|u(x) - u(y) \right|^2~dydx ds \leq e^{\frac{C}{\beta}t }\mathcal F(f_0)
\end{eqnarray*}
for all $t>0$.
\end{theorem}


\section{A priori estimates and velocity averages}
In this section, we collect some 
results that we will need for the proof of Theorem~\ref{thm:main}.  
Since the only interesting case is when $\beta>0$, we will take 
$$\beta=1$$
throughout the proof.
First, we derive some priori estimates satisfied by solutions of \eqref{eq:eq}
(when introducing the regularized equation in Section \ref{sec:approximate}, we will make sure that these estimates still hold).
Then, we recall the classical averaging lemma which will play a crucial role in the proof.

\subsection{A priori estimates}
Any smooth solution of (\ref{eq:eq}) satisfies the following conservation of mass:
$$
\int_{\R^{2d}} f(x,v,t)  \, dx \, dv= \int_{\R^{2d}} f_0(x,v,t) \, dx \, dv =: M.
$$
Since  $f_0\geq 0$,  we have $f\geq 0$,  and so the conservation of mass implies a a priori bound in $L^\infty(0,T;L^1(\R^{2d}))$.
It is well known that solutions of (\ref{eq:eq}) also satisfy a a priori estimates in $L^\infty(0,T;L^p(\R^{2d}))$ (cf.~\cite{HT-2008}) for all $p\in [1,\infty]$. 
More precisely:
\begin{lemma}\label{prop:est}
Let $f$ be a smooth solution of (\ref{eq:eq}), then
\begin{equation}\label{eq:Lp} 
\|f\|_{L^\infty(0,T;L^p(\R^{2d}))} + \sigma \|\Grad_v f^\frac{p}{2}\|_{L^2((0,T)\times \R^d\times \R^d)}^\frac{2}{p} \leq e^{CT/p'}  \|f_0\|_{L^p(\R^{2d})}
\end{equation}
with $C=d [1+\|K_0\|_{L^\infty} M] $ and $p'=\frac{p}{p-1}$, for all $p\in [1,\infty]$.
\end{lemma}
\begin{proof}
For $p<\infty$, a simple computation yields
\begin{eqnarray*}
\frac{d}{dt} \int f^p\, dx\, dv & = & -\frac{4(p-1)}{p} \sigma \int |\na( f^{p/2})|^2\, dx\, dv \\
&& -(p-1) \int f^p \Div_v L [f]\, dx\, dv+ (p-1)d \int f^p  \, dx\, dv
\end{eqnarray*}
Since
$$ \Div_v L [f] =- d\int K_0(x,y) f(y,w)  \, dy\, dw$$
we also have
$$ |\Div_v L [f] |\leq d \|K_0\|_{L^\infty} M.$$
We deduce:
\begin{eqnarray*}
\frac{d}{dt} \int f^p\, dx\, dv & = & -\frac{4(p-1)}{p} \sigma\int |\na( f^{p/2})|^2\, dx\, dv \\
&& + (p-1) d [1+\|K_0\|_{L^\infty} M]  \int f^p \, dx\, dv
\end{eqnarray*}
which implies  (\ref{eq:Lp}) by a Gronwall argument.
\end{proof}

Next, we will need  better integrability of $f$ for large $v$ and $x$. 
Let us denote
$$ \E(f) = \int \frac{v^2}{2} f + \Phi(x) f\, dx\, dv.$$
We then have the following important a priori estimate:
\begin{lemma}\label{lem:apriori}
Let $f$ be a smooth solution of (\ref{eq:eq}), then
\begin{eqnarray}
&& \frac{d}{dt} \E(f) +   \int_{\R^{2d}} |u-v|^2 f\, dx\, dv \nonumber \\
&& \qquad + \frac{1}{2}\int_{\R^{2d}}\int_{\R^{2d}}K_0(x,y)f(x,v)f(y,w)\left|v - w\right|^2~dwdydvdx\nonumber \\
&& \qquad\qquad  = \sigma d \int f\, dx\, dv =  \sigma d M.\label{eq:apriori0}
\end{eqnarray}
In particular, 
\begin{equation}\label{eq:apriori} 
\E(f(t)) \leq  \sigma d M t + \E(f_0).
\end{equation}
\end{lemma}
\begin{remark}
Inequality (\ref{eq:apriori0}) will play a  fundamental role in this paper. It implies that the individuals remain somewhat localized in space and velocity for all time. 
We note that it is not quite the standard entropy inequality when $\sigma>0$ (the natural entropy inequality for \eqref{eq:eq} when $\sigma>0$, which involves a  term of the form $f\log f$, will be detailed in Section \ref{sec:entropy}. It is not needed for the proof of Theorem \ref{thm:main}). 

Our existence result can be generalized to other models, provided we can still establish (\ref{eq:apriori}).
For instance, self-propulsion and friction can be taken into account, via a term of the form
$ -\Div((a-b|v|^2)vf)$ in the right hand side of (\ref{eq:eq}). 
It is easy to check that (\ref{eq:apriori}) then becomes
$$ \E(f)\leq \left[\E(f_0)+ \frac{\sigma d M}{2a}\right]e^{2at}.$$

Another important generalization is to consider non-symmetric flocking interactions (i.e. when $K_0(x,y)\neq K_0(y,x)$). The derivation of a bound on $\E(f)$  requires stronger assumptions on the convolution kernel, and these assumptions are difficult to check in the case of Motsch-Tadmor model (because of the normalization of the convolution kernel). 
This will be discussed in Section \ref{sec:MT}.
\end{remark}

\begin{proof}[Proof of Lemma \ref{lem:apriori}]
Using Equation (\ref{eq:eq}), we compute
\begin{equation}\label{eq:ap1}
	\begin{split}
		\frac{d}{dt}\E(f) &= 
	  \int_{\R^{2d}} \left(\Phi(x) + \frac{|v|^2}{2}\right) \pa_t f~dvdx \\
		& = \int_{\R^{2d}}fL[f]v + f(u-v)\cdot v ~dvdx + \sigma d \int f\, dx\, dv \\
		& = -\frac{1}{2}\int_{\R^{4d}}K_0(x,y)f(x,v)f(y,w)|v-w|^2~dwdydvdx \\
		&\qquad - \int_{\R^{2d}}f|u-v|^2~dvdx + \sigma d \int f\, dx\, dv
	\end{split}
\end{equation}
where we used the symmetry $K_0(x,y)=K_0(y,x)$ and the fact that $\int f(u-v)\, dv =0$.
The lemma follows.
\end{proof}

Finally, we recall the following lemma which 
will be proven to be  very useful in the upcoming analysis:
\begin{lemma}\label{lem:jn}
Assume that $f$ satisfies
$$ \|f \|_{L^{\infty}([0,T]\times \R^{2d})} \leq M, \qquad \int_{\R^{2d}} |v|^2 f \, dv dx \leq M .$$
Then there exists a constant $C=C(M) $ such that
\begin{equation}\label{eq:jnest}
\begin{array}{l}
\|\rho\|_{L^\infty(0,T;L^{p}(\R^d))} \leq C ,\quad\mbox{ for every $p\in[1,\frac{d+2}{d}),$}\\ [5pt]
\| j\|_{L^\infty(0,T;L^{p}(\R^d))} \leq C,\quad\mbox{ for every $p\in[1,\frac{d+2}{d+1}),$}
\end{array}
\end{equation}
where $\rho=\int f\, dv$ and $j=\int vf\,dv $.
\end{lemma}
\begin{proof}
Let $p\in (1,\infty)$ and let  $q$ be such that $1/p+1/q=1$. Then  we have:
\begin{eqnarray*}
\rho(x,t) &= & \int (1+|v|)^{2/p}f^{1/p}(v) \frac{f^{1/q}}{(1+|v|)^{2/p}}\, dv \\
&\leq & \left(\int (1+|v|)^{2}f(v)\, dv \right)^{1/p} \left(\int \frac{f(v)}{(1+|v|)^{2q/p}}\, dv\right)^{1/q} .
\end{eqnarray*}
In particular, if $2q/p>d$, we deduce
$$\rho(x,t) \leq C\|f(t)\|_{L^{\infty}}^{1/q} \left(\int (1+|v|)^{2}f(v)\, dv \right)^{1/p} $$
and so
$$ \|\rho(t)\|_{L^{p}}^{p}= \int \rho(x,t)^{p} \, dx \leq C \int \int (1+|v|)^{2}f(v)\, dv \, dx .$$ 
Noting that the condition $2q/p>d$ is equivalent to 
$ p < \frac{d+2}{d},$ this implies the first inequality in (\ref{eq:jnest}).

A similar argument holds for $j$:
\begin{eqnarray*}
j (x,t) &\leq & \int (1+|v|)^{2/p}f^{1/p}(v) \frac{f^{1/q}}{(1+|v|)^{2/p-1}}\, dv \\
&\leq & \left(\int (1+|v|)^{2}f(v)\, dv \right)^{1/p} \left(\int \frac{f(v)}{(1+|v|)^{2q/p-q}}\, dv\right)^{1/q}. 
\end{eqnarray*}
In particular, if $2q/p-q>d$, we deduce
\begin{equation}\label{eq:trengerikkefyfaen}
j (x,t) \leq C\|f(t)\|_{L^{\infty}}^{1/q} \left(\int (1+|v|)^{2}f(v)\, dv \right)^{1/p}	
\end{equation}
and so
$$ \|j (t)\|_{L^{p}}^{p}= \int n(x,t)^{p} \, dx \leq C \int \int (1+|v|)^{2}f(v)\, dv \, dx ,$$ 
where the condition $2q/p-q>d$ is equivalent to
$$ p < \frac{d+2}{d+1}.$$
\end{proof}

\subsection{Velocity averaging and compactness}
In the proof of Theorem \ref{thm:main}, we will need
 some 
 compactness results for the density $\vr = \int f~dv$
and the first moment $j = \int fv~dv$ of sequences of approximated solutions.
Such compactness will be obtained by using the  celebrated velocity averaging lemma
for  quantities of the form
\begin{equation*}
	\vr_\psi = \int_{\R^{2d}}f\psi(v)~dv,
\end{equation*}
where $\psi$ is a locally supported function together with the bound  (\ref{eq:apriori}). 

We first recall the following result (see Perthame and Souganidis \cite{PS-1998}):

\begin{proposition}\label{pro:velocity}
Let $\{f^n\}_n$ be bounded in $L_\text{loc}^p(\R^{2d+1})$ with $1 < p < \infty$, 
and $\{G^n\}_n$ be bounded in $L_\text{loc}^p(\R^{2d+1})$. If 
$f^n$ and $G^n$ satisfy
$$
f^n_t + v\cdot \Grad_x f^n = \Grad_v^k G^n, \qquad f^n|_{t=0} = f^0 \in L^p(\R^{2d}),
$$
for some multi-index $k$ and $\psi \in C^{|k|}_\text{c}(\R^{2d})$,
then $\{\vr_\psi^n\}$ is relatively compact in $L^p_\text{loc}(\R^{d+1})$.
\end{proposition}

The velocity averaging lemma cannot be directly applied to  conclude compactness 
of $\vr$ and $j$ since the function $\psi$ is required to 
be compactly supported.  
Also, we would like to get compactness in $L^p(\R^{d+1})$ instead of $L^p_\text{loc}(\R^{d+1})$.
We will thus need the following lemma, consequence of the decay of $f$ for large $x,v$ (provided by  
Lemma \ref{lem:apriori}):
\begin{lemma}\label{cor:velocity}
Let $\{f^n\}_n$ and $\{G^n\}_n$ be as in Proposition \ref{pro:velocity} 
and assume that
\begin{equation*}
f^n \mbox{ is bounded in } L^\infty(\R^{2d+1}),
\end{equation*}
\begin{equation*}
 (|v|^2+\Phi)f^n \mbox{ is bounded in } L^\infty(0,T;L^1(\R^{2d+1})).
\end{equation*}
Then, for any $\psi(v)$ such that $|\psi(v)| \leq c|v|$ and $q < \frac{d+2}{d+1}$, the sequence
\begin{equation}\label{eq:sec}
	\Set{\int_{\R^d}f^n \psi(v)~dv}_n,
\end{equation}
is relatively compact in $L^q((0,T)\times \R^d)$.
\end{lemma}

\begin{proof}\label{pf:}
	1. We first prove compactness of the sequence $\Set{\int_{\R^d}f^n \psi(v) ~dv}_n$ in $L^q_\text{loc}$.
	Since $\psi$ is not compactly supported, we consider $\vphi_k(v)$, a sequence of smooth functions satisfying
	$$ \vphi_k(v) =1 \mbox{ for $|v|\leq k$ and } \vphi_k(v)=0 \mbox{ for $|v|\geq k+1$}.$$ 
	Proposition \ref{pro:velocity} then implies that for all $k\in \N$, the sequence $m^{k,n}_\psi = \int \vphi_k f^n\psi(v)\, dv$ converges 
	strongly in $L^q_{\text{loc}}((0,T)\times\R^d)$ (up to a subsequence) to some 
	$m^k_\psi$.	
	Now, for $k_1 > k_2$ and any $0 < \alpha < 1$, we have that
	\begin{equation*}
		\begin{split}
				 |m_\psi^{k_1,n} - m_\psi^{k_2,n}| \leq C\int_{|v|\geq k_2 + 1} f^n|\psi(v)|~dv 
				&\leq \frac{C}{(k_2)^\alpha}\int_{\R^d}f^n|v|^{1+\alpha}~dv\\
				& \leq \frac{C}{k_2^\alpha}(\vr^n)^\frac{(1-\alpha)}{2}\left(\int_{\R^d}f^n|v|^2~dv\right)^\frac{(1+\alpha)}{2}.
		\end{split}
	\end{equation*}
	By integrating this inequality over space and applying of the H\"older inequality,
	\begin{equation*}
	\begin{split}
		\|m_\psi^{k_1,n} - m_\psi^{k_2,n}\|^q_{L^q(\R^{2d})} 
		&\leq \frac{C}{(k_2)^{q\alpha}} \left(\int_{\R^d}(\vr^n)^\frac{q(1-\alpha)}{2-q(1+\alpha)}~dx\right)^\frac{2-q(1+\alpha)}{2}\left(\int_{\R^{2d}}f^n|v|^2~dvdx\right)^\frac{q(1+\alpha)}{2}.
	\end{split}
	\end{equation*}
	Lemma \ref{lem:jn} implies that $\vr^n$ is bounded in $L^\infty(0,T;L^p(\R^d))$ for $p\in[1,\frac{d+2}{d})$, and so the right hand side above is bounded if $q$ is such that
	\begin{equation*}
		\frac{q(1-\alpha)}{2-q(1+\alpha)} < \frac{d+2}{d} \quad \Leftrightarrow \quad  q< \frac{d+2}{d+1 +\alpha}.
	\end{equation*}
	For any $q\in (1,\frac{d+2}{d+1})$, we can thus choose $\alpha$ small enough so that 
$$	\|m_\psi^{k_1,n} - m_\psi^{k_2,n}\|^q_{L^q(\R^{2d})}  \leq C\frac{1}{k_2^{q\alpha}}.$$
Passing to the limit $n\to\infty$,  we conclude that the sequence
	$\{m^{k}_\psi\}$ is Cauchy in $L^q_{\text{loc}}((0,T)\times\R^d)$ and thus converges to $m_\psi$.
	By a diagonal extraction process, we deduce the existence of a subsequence along which
	\begin{equation*}
		m^{n,n}_\psi = \int_{\R^d}f^n\psi(v)\vphi_n(v)~dv \quad\longrightarrow \quad m_\psi 	\end{equation*}
	in $L^q_{\text{loc}}((0,T)\times\R^d)$ (for any $q < (d+2)/(d+1)$).
	
	Finally, we have
	\begin{align}
	|\int   f^n\psi(v)\, dv -m_\psi| 
	 &\leq   \left|\int (1-\vphi_n)  \psi(v) f^n\, dv\right| + |m^{n,n}_\psi -m_\psi| \nonumber \label{eq:redo} \\
	 & \leq   \left|\int_{|v|\geq n+1}   |\psi(v)|f^n\, dv\right| + |m^{n,n}_\psi -m_\psi| \\
	 & \leq  \frac{1}{n^\alpha} \left|\int |v|^{1+\alpha}  f^n\, dv\right| +|m^{n,n}_\psi -m_\psi|,\nonumber
	 \end{align}
	which converges to zero in $L^q_{\text{loc}}((0,T)\times\R^d)$ for $q<(d+2)/(d+1)$ (in particular, the first term in the right hand side is bounded in $L^q$ for the same reason as above).
	\vspace{10pt}
	
	2. We have thus established the compactness in $L^q_\text{loc}((0,T)\times\R^d)$. To prove compactness in $L^q((0,T)\times\R^d)$, we argue as above, but instead of using the fact that $\int_{\R^{2d}} |v|^2 f^n\, dx\, dv$ is bounded, we need to show that $\int \psi(v)f^n\, dv$ decay for $|x|\rightarrow \infty$.
		To prove this, we proceed as in the proof of Lemma \ref{lem:jn} (see in particular \eqref{eq:trengerikkefyfaen}), to show that for $l< 2$ and $p < \frac{d+l}{d+1}$,
	\begin{equation*}
		\begin{split}
			\left|\int_{\R^d}\psi(v)f^n~dv\right|^p &\leq 
			\|f^n\|_{L^\infty}^\frac{p}{q}\left(\int_{\R^d}(1+|v|^l)f^n~dv\right) \\			
			&\leq C\left(\vr^n +  (\vr^n)^\frac{2-l}{2}\left(\int_{\R^d}|v|^2f^n~dv\right)^\frac{l}{2}\right).
		\end{split}
	\end{equation*}
	By integrating  for $|x|\geq k$, we then see that
	\begin{equation*}
		\begin{split}
			&\int_{|x|\geq k}\left|\int_{\R^d}\psi(v)f^n~dv\right|^p~dx \\
			&\leq \frac{C}{\Phi(k)}\int_{|x|\geq k}\vr^n \Phi~dx 
			+ \frac{C}{\Phi(k)^\frac{2-l}{2}}\left(\int_{|x|\geq k}\vr^n \Phi~dx \right)^\frac{2-l}{2}\left(\int_{|x|\geq k}\int_{\R^d}f^n|v|^2 ~dvdx \right)^\frac{l}{2}\\
			&\leq C\left(\frac{1}{\Phi(k)} + \frac{1}{\Phi(k)^\frac{2-l}{2}}\right).
		\end{split}
	\end{equation*}
	In particular, for any $q < \frac{d+2}{d+1}$, we can choose $l<2$ such that the inequality above holds with $p=q$, and so 
$$	\int_{|x|\geq k}\left|\int_{\R^d}\psi(v)f^n~dv\right|^q~dx \longrightarrow 0 \quad \mbox{ as } k\to \infty \mbox{ uniformly w.r.t. $n$}.$$
We can now proceed as in the first part of the proof to show that the sequence
$\int_{\R^d}\psi(v)f^n~dv$ converges in $L^q((0,T)\times\R^d)$.

\end{proof}

\section{Approximate solutions}\label{sec:approximate}
In this section, we prove the existence of solutions for an approximated equation (by a fixed point argument).
In the next section, we pass to the limit in the approximation to obtain a solution of \eqref{eq:eq} and prove Theorem \ref{thm:main}.

As pointed out in the introduction, the main difficulty in (\ref{eq:eq}) is the lack of estimates on the velocity $u=\frac{\int vf\,fv}{\int f\,dv}$.
We thus consider the following equation, in which the velocity term $u$ has been regularized:
\begin{equation}\label{eq:app}
\left\{
\begin{array}{ll}	
\pa_t f + v \cdot \Grad_x f - \Div_v(f\Grad_x \Phi) + \Div_v\left(fL [f]\right) = \sigma \Delta_v f - \Div_v(f(\chi_\lambda(u_\delta) -v )) 
\\[5pt]
f(x,v,0)=f_0(x,v),
\end{array}
\right.
\end{equation}
where 
\begin{itemize}
\item the function $\chi_\lambda$ is the truncation function 
$$\chi_\lambda(u) = u\, 1_{|u|\leq \lambda}.$$
\item  $u_\delta$ is  defined by:
\begin{equation}\label{eq:uu'}
u_\delta = \frac{\int_{\R^d} v  f\, dv}{\delta + \int_{\R^d} f\, dv} = \frac{\rho}{\delta+\rho} u.
\end{equation}
\end{itemize}
Formally, we see that we recover \eqref{eq:eq} in the limit $\delta \rightarrow  0$ and $\lambda \rightarrow \infty$.
The rigorous arguments for taking these limits are given in the ensuing section. In this section, 
we prove the following existence result for fixed $\delta$ and $\lambda$. 
More precisely, we prove:
\begin{proposition}\label{prop:ex}
Let $f_0 \geq 0$ satisfy the condition of Theorem \ref{thm:main}.
Then, for any $\delta>0$, $\lambda>0$ there exists a solution $f \in \mathcal C(0,T;L^1(\R^{2d}))$ of (\ref{eq:app}) satisfying 
\begin{equation}\label{eq:Lpapp} 
\|f\|_{L^\infty(0,T;L^p(\R^{2d}))} + \sigma \|\Grad_v f^\frac{p}{2}\|_{L^2((0,T)\times \R^d\times \R^d)}^\frac{2}{p} \leq e^{CT/p'}  \|f_0\|_{L^p(\R^{2d})}
\end{equation}
for all $p\in[1,\infty]$, and 
\begin{equation}\label{eq:Eapp} 
\sup_{t\in [0,T)}\E(f) \leq \E(f^0)+ \sigma d M T.
\end{equation}
\end{proposition}
The proof of Proposition \ref{prop:ex} relies on a fixed point argument: Fix
$$p_0 \in \left(1,\frac{d+2}{d+1}\right)$$
and for a given $\bar u\in L^{p_0}(0,T;L^{p_0}(\R^d))$ let $f$ be the solution of 
\begin{equation}\label{eq:fp1}
\left\{
\begin{array}{l}	
\pa_t f + v \cdot \Grad_x f - \Div_v(f\Grad_x \Phi) + \Div_v\left(fL [f]\right) = \sigma \Delta_v f - \Div_v(f(\chi_\lambda(\bar u) -v ))\\[5pt]
f(x,v,0)=f_0(x,v). 
\end{array}
\right.
\end{equation}
We then consider the mapping 
\begin{equation}\label{def:t}
	\bar u \mapsto T(\bar u):= u_\delta = \frac{\int_{\R^d} v  f\, dv}{\delta + \int_{\R^d} f\, dv}=\frac{\rho u }{\delta+\rho}.
\end{equation}
In the remaining parts of this section, we prove the existence of a fixed point 
for the mapping 
$ T: L^{p_0}(0,T;L^{p_0}(\R^d))\longrightarrow L^{p_0}(0,T;L^{p_0}(\R^d))$
and thereby prove Proposition \ref{prop:ex}. 

\subsection{The operator $T$ is well-defined}
First, we need to check that the operator $T$ is well-defined, and to 
derive some bounds on $f$.  
We start with the following result:
\begin{lemma}\label{prop:estap}
For all $\bar u \in L^p(0,T;L^p(\R^d))$, there exists a unique $f\in \mathcal C(0,T;L^1(\R^d))$ solution of (\ref{eq:fp1}). Furthermore, $f$ is non-negative and  satisfies
\begin{equation}\label{eq:Lpap} 
\|f\|_{L^\infty(0,T;L^p(\R^{2d}))} +\sigma \|\Grad_v f^\frac{p}{2}\|_{L^2((0,T)\times \R^d\times \R^d)}^\frac{2}{p} \leq e^{CT/p'}  \|f_0\|_{L^p(\R^{2d})},
\end{equation}
with $C=d [1+\|K_0\|_{L^\infty} M] $ and $p'=\frac{p}{p-1}$, for all $p\in [1,\infty]$ and 
\begin{equation}\label{eq:vmapp}
	\begin{split}
		 \sup_{t\in [0,T]} \E(f(t)) + \frac{1}{2}\int_0^T\! \! \!\!\int_{\R^{2d}}\!\!  fv^2\, dx\, dv\, dt & \leq \frac{1}{2}\int_0^T\!\! \!\!\int_{\R^{2d}}f|\chi_{\lambda}(\bar u)|^2~dvdxdt  +\sigma dMT \\
		 &\leq \frac{\lambda^2 MT}{2} +\sigma dM  T.
	\end{split}
\end{equation}

\end{lemma}
\begin{proof}
Since $\chi_\lambda (\bar u)\in L^\infty((0,T)\times\R^d)$, the existence of a solution to (\ref{eq:fp1}) is classical (we prove it in Section \ref{sec:other} for the sake of completeness, see Theorem \ref{thm:twofixed}). Moreover, since 
$\chi_\lambda(\bar u)$ is independent of $v$, we have that 
\begin{equation*}
	\int_{\R^{d}} f \chi_\lambda(\bar u) \Grad_v f^{p-1}~dv  = -\frac{1}{p}\int_{\R^d}f^p \Div_v \chi_\lambda(\bar u)~dv = 0.
\end{equation*}
Hence, we can perform the same computations as in Lemma \ref{prop:est} to obtain \eqref{eq:Lpap}.

Next, using the equation \eqref{eq:fp1}, we write
\begin{equation*}
	\begin{split}
		\frac{d}{dt}\E(f) &= 
		\frac{d}{dt} \int_{\R^{2d}} \left(\Phi + \frac{|v|^2}{2}\right) f~dvdx \\
		& = \int_{\R^{2d}}fL[f]v + f\chi_\lambda(\bar u) - f|v|^2~dvdx +\sigma d\int f\, dv\, dx\\
		& = -\frac{1}{2}\int_{\R^{4d}}K_0(x,y)f(x,v)f(y,w)|v-w|^2~dwdydvdx - \int_{\R^{2d}}f|v|^2~dvdx \\
		&\qquad + \int_{\R^{2d}}f\chi_{\lambda}(\bar u)v~dvdx + \sigma d M.
	\end{split}
\end{equation*}
Finally, writing
\begin{equation*}
	\left|\int_{\R^{2d}}f\chi_{\lambda}(\bar u)v~dvdx\right| \leq \frac{1}{2}\int_{\R^{2d}}f|\chi_{\lambda}(\bar u)|^2~dvdx + \frac{1}{2}\int_{\R^{2d}}f|v|^2~dvdx,
\end{equation*} 
we deduce \eqref{eq:vmapp}.
%
%
%
\end{proof}

The following lemma establishes the continuity of the operator  $T.$
\begin{lemma}\label{lem:ft1}
The operator $T$ is  continuous  and  
there exists $C(\delta,\lambda)$ such that for all $\bar u\in L^{p_0}(0,T;L^{p_0}(\R^d))$, 
$$ \|T(\bar u)\|_{L^{p_0}(0,T;L^{p_0}(\R^d))} \leq C(\delta, \lambda)$$
\end{lemma}
\begin{proof}
We have $|T(\bar u)|\leq\frac{1}{\delta}|j|$, and we recall that  ${p_0}< \frac{d+2}{d+1}$. 
So lemma \ref{lem:jn} together with (\ref{eq:vmapp}) implies that there exists a constant $C$ such that
$$\|j\|_{L^\infty(0,T;L^{p_0}(\R^d))} \leq  C,
$$
where $C$ only depends on $\|f_0\|_{L^\infty}$ and $\int |v|^2 f_0\, dx\,dv$. 
\end{proof}
\vspace{5pt}

\subsection{The operator $T$ is compact}

Compactness of the operator $T$ follows from the following lemma:
\begin{lemma}\label{lem:ft3}
Let $\{\bar u^n\}_{n\in\N}$ be a bounded sequence in $ L^{p_0}(0,T;L^{p_0}(\R^d))$. Then up to a subsequence, $T(\bar u^n)$ converges strongly in $L^{p_0}(0,T;L^{p_0}(\R^d))$.
\end{lemma}
\begin{proof}
By definition of $T$, we have that
\begin{equation*}
	T(\bar u^n)= \frac{j^n}{\delta + \vr^n}.
\end{equation*}
So in order  to prove that the sequence $T(\bar u^n)$ is relatively compact in $L^{p_0}(0,T;L^{p_0}(\R^d))$, we have to show that 
$\vr^n$ converges a.e and that $j^n$ 
is relatively compact in $L^{p_0}(0,T;L^{p_0}(\R^d))$.
This follows from Lemma \ref{cor:velocity}. Indeed, let
$G^n=f^n\na_x \Phi  +\sigma\na_v f^n- f^n(\chi_\lambda(\bar u^n)-v)-f^nL[f^n]$
then \eqref{eq:app} can be rewritten as 
\begin{equation*}
	f^n_t + v\cdot \Grad_v f^n = \Div_v G^n.
\end{equation*}
Furthermore, we have the following lemma (whose proof is postponed to the end of this section):
\begin{lemma}\label{lem:G}
For any $q\leq 2$,  there exists a constant $C$ independent of $n$ such that
\begin{equation}\label{eq:Gn}
\|G^n\|_{L^\infty(0,T;L^q(\R^d\times \R^d))}  \leq C \qquad \mbox{ for all $n\geq 0$}.
\end{equation}
\end{lemma}

In particular, Sobolev embeddings imply that 
 $G^n$ is relatively compact in $W_\text{loc}^{-1,r}((0,T)\times \R^{2d})$, for  $r < \frac{2(2d+1)}{2d -1 }$,
and hence also in $W_\text{loc}^{-1,p_0}((0,T)\times \R^{2d})$. In view of (\ref{eq:Lpapp}) and (\ref{eq:vmapp}), we can apply Lemma~\ref{cor:velocity}  (with $\psi(v)=1$ and $\psi(v)=v$)
to conclude the existence of functions $\vr$ and $j$ such that (up to a subsequence)
\begin{equation*}
	\begin{split}
		\vr^n=\int_{\R^d} f^n\, dv &\conv \vr \quad \text{ in $L^{p_0}((0,T)\times\R^d)$ and a.e.} , \\
		j^n=\int_{\R^d} v f^n\, dv &\conv j \quad \text{ in $L^{p_0}((0,T)\times\R^d)$ }.
	\end{split}
\end{equation*}
\end{proof}

\begin{proof}[Proof  of Lemma \ref{lem:G}]
By repeated applications of the H\"older inequality (we drop the $n$ dependence for the sake of clarity),
\begin{equation*}
	\begin{split}
		\|G(t)\|_{L^q(\R^{2d})}&\leq 		
		 \|\na \Phi\|_{L^\infty(\R^{2d})} \|f\|_{L^q(\R^{2d})} + \|\na_v f \|_{L^q(\R^{2d}) }
		+ C\lambda \|f\|_{L^q (\R^{2d})}\\
		&+\|K_0\|_{L^\infty(\R^{2d})} \big(\|f\|_{L^1(\R^{2d})} \|vf\|_{L^q(\R^{2d})}+\|f\|_{L^q(\R^{2d})} \|vf \|_{L^1(\R^{2d})}\big),
	\end{split}
\end{equation*}
where we have that
$$\|vf\|_{L^q} \leq \left(\int |v|^2 f\, dv\,dx \right)^{q/2} \|f\|_{L^{\frac{q}{2-q}}}^{q/2} \qquad \mbox{ for $q\in[1,2)$},
$$
and 
$$\|vf\|_{L^2} \leq  \|f\|_{L^{\infty} }\, \int |v|^2 f\, dv\,dx,
$$
which are both bounded by Lemma \ref{prop:estap}.

It remains to bound the term involving $\Grad_v f$. For this purpose, we first 
observe that  \eqref{eq:Lp} with $p=1$ provides the bound
\begin{equation}\label{eq:j2}
	\int_0^T\int_Q \frac{1}{f}|\Grad_v f|^2~dvdxdt \leq CM.
\end{equation}
Using this together with the H\"older inequality and $q=\frac{2p}{p+1}$,  we get that
\begin{equation}\label{eq:j3}
	\begin{split}
		\int_0^T \int_Q |\Grad_v f|^q~dvdxdt &= \int_0^T\int_Q f^{\frac{q}{2}}f^{-\frac{q}{2}}|\Grad_v f|^q~dvdxdt \\
		&\leq \int_0^T \|f\|_{L^\frac{q}{2-q}}^\frac{q}{2}\left(\int_{Q}\frac{1}{f}|\Grad_v f|^2~dvdx\right)^\frac{q}{2}~dt \\
		&= \int_0^T \|f\|_{L^p}^\frac{p}{p+1}\left(\int_{Q}\frac{1}{f}|\Grad_v f|^2~dvdx\right)^\frac{p}{p+1}~dt \leq C,
	\end{split}
\end{equation}
So using \eqref{eq:Lpap} and \eqref{eq:vmapp}, we conclude the proof.
\end{proof}

\subsection{Proof of Proposition \ref{prop:ex}}
Lemma \ref{lem:ft1} and \ref{lem:ft3}, together with Schauder fixed point theorem imply the existence of a fixed point $u\in L^{p_0}((0,T)\times\R^d)$ of $T$.
The corresponding solution of (\ref{eq:fp1}) solves (\ref{eq:app}).  
Furthermore, it is readily seen that (\ref{eq:Lpap}) implies (\ref{eq:Lpapp}).
Finally, since 
\begin{equation}\label{eq:ud}
| u_\delta| = \frac{\left|\int_{\R^d} v  f\, dv\right|}{\delta + \int_{\R^d} f\, dv}\leq| u| = \frac{\left|\int_{\R^d} v  f\, dv\right|}{\int_{\R^d} f\, dv},	
\end{equation} 
we have
$$ \int f |\chi_\lambda (u_\delta )|^2 dx\, dv \leq \int \rho |u|^2 dx \leq \int f|v|^2\, dxdv$$
and so (\ref{eq:vmapp}) yields (\ref{eq:Eapp}).

\vspace{15pt}

\section{Existence of solutions (Proof of Theorem \ref{thm:main})}\label{sec:existence}
For $\delta,\lambda>0$, we denote  by $f_{\delta, \lambda}$ the solution of \eqref{eq:app}  given by Proposition \ref{prop:ex}. We also denote 
$$\rho_{\delta, \lambda}=\int_{\R^d}f_{\delta, \lambda}\, dv\quad \mbox{ and } j_{\delta, \lambda}=\int _{\R^d}vf_{\delta, \lambda}\, dv$$
and 
$$ u _{\delta, \lambda}^\delta = \frac{j_{\delta, \lambda}}{\delta +\rho_{\delta, \lambda} }.
$$
In this section we show how to pass to 
the limit $\lambda\to\infty$ and $\delta\to 0$, thereby proving Theorem \ref{thm:main}.
First, We recall that Proposition \ref{prop:ex} implies the following bounds (which are uniform with respect to  $\delta$ and $\lambda$):
\begin{corollary}\label{cor:entropy}
There exists $C$ independent of $\delta$ and $\lambda$ such that
\begin{equation}\label{eq:Lpn} 
|| f_{\delta, \lambda}||_{L^p(0,T;L^p(\R^{2d}))} \leq e^{CT/p'}  || f_0 ||_{L^p(\R^{2d})}, 
\end{equation}
and
\begin{equation}\label{eq:entropyn} 
 \int_{\R^d}\int_{\R^d} f_{\delta, \lambda}\left( \frac{|v|^2}{2}  + \Phi\right)~dvdx \leq C.
 \end{equation}
In particular, Lemma \ref{lem:jn} implies
$$ ||\rho_{\delta, \lambda}||_{L^\infty(0,T;L^p(\R^d))} \leq C \quad \mbox{for all $p<\frac{d+2}{d}$}.$$
$$ ||j_{\delta, \lambda} ||_{L^\infty(0,T;L^p(\R^d))} \leq C \quad \mbox{for all $p<\frac{d+2}{d+1}$}.$$
\end{corollary}

Finally, in order to use the averaging Lemma \ref{cor:velocity}, we rewrite \eqref{eq:app}  as
\begin{equation}\label{eq:fn} 
\pa_t f_{\delta, \lambda}  + v\cdot \na_x f _{\delta, \lambda} = \Div_v G_{\delta, \lambda} 
\end{equation}
with
$$G_{\delta, \lambda} =f_{\delta, \lambda} \na \Phi  +\sigma \na_v f_{\delta, \lambda} - f_{\delta, \lambda} (\chi_\lambda(u^\delta_{\delta, \lambda} )-v)-f_{\delta, \lambda} L[f_{\delta, \lambda} ] ,$$
and we will need the following result:
\begin{lemma}\label{lem:Gf}
For any $q\leq2$,  there exists a constant $C$ independent of $\delta$ and $\lambda$ such that
\begin{equation}\label{eq:Gn}
||G_{\delta, \lambda}||_{L^\infty(0,T;L^q(\R^d\times \R^d))}  \leq C .
\end{equation}
\end{lemma}
\begin{proof}
This lemma is similar to Lemma \ref{lem:G}.
The only additional difficulty is to bound the term $f_{\delta, \lambda}\chi_\lambda(u^\delta_{\delta, \lambda} )$ uniformly with respect to $\lambda$.
But we note that $f_{\delta, \lambda}\chi_\lambda(u^\delta_{\delta, \lambda})\leq f_{\delta, \lambda} u_{\delta, \lambda}$ and 
$$||f_{\delta, \lambda} u_{\delta, \lambda}||_{L^\infty(0,T;L^2(\R^{2d}))}  \leq ||f_{\delta, \lambda} ||_{L^\infty((0,T)\times\R^{2d})} \int  f_{\delta, \lambda} |v|^2\, dx\,dv $$
and so  (\ref{eq:Lpn}) and (\ref{eq:entropyn}) imply the existence
of a constant $C > 0$, independent of $\lambda$
and $\delta$, such that
$$
||f_{\delta, \lambda} u_{\delta, \lambda}||_{L^\infty(0,T;L^2(\R^{2d}))} \leq C
$$
which concludes the proof.
\end{proof}
\vspace{10pt}

\subsection{Limit as $\lambda\to \infty$}
We now fix $\delta>0$ and consider a sequence $\lambda_n\to \infty$. We denote by $f^n=f_{\delta, \lambda_n}$ the corresponding solution of (\ref{eq:app}) and
$$u_\delta^n = \frac{\int v f^n\, dv}{\delta+\int f^n\, dv} = \frac{j^n}{\delta+\rho^n}.$$
We then have:
\begin{lemma}\label{lem:lambda}
Up to a subsequence,  $f^n$ converges weakly in $\star-L^\infty(0,T;L^1(\R^{2d})\cap L^\infty(\R^{2d}))$,  to some function $f$, and $u^n_\delta $ converges strongly to  $u_\delta=\frac{\int fv\, dv}{\delta+\int f\,dv}$ in $L^p((0,T)\times \R^{2d})$, $p < \frac{d+2}{d+1}$.

Furthermore, $f$ is a weak solution of \eqref{eq:app} with $\lambda = \infty$, and it satisfies   the a priori estimates of Corollary \ref{cor:entropy} and Lemma \ref{lem:Gf}.\end{lemma}

\begin{proof}
By virtue of \eqref{eq:Lpn}, there exists a function $f \in L^\infty(0,T;L^1(\R^{2d})\cap L^\infty(\R^{2d}))$ 
such that, up to a subsequence,  $f^n \weakstar f$ in $L^\infty(0,T;L^1(\R^{2d})\cap L^\infty(\R^{2d}))$.

Furthermore, using (\ref{eq:fn}) and Lemma \ref{lem:Gf}, we can reproduce the arguments of Lemma \ref{lem:ft3}
to show that
 $\vr^n$ and $j^n$ converge strongly and almost everywhere  to $\vr$ and $j$ in $L^p((0,T)\times \R^{2d})$ for $p < \frac{d+2}{d+1}$.
We deduce:
 $$
	u^n_\delta \rightarrow u_\delta:= \frac{j}{\delta + \vr} \text{ in $L^p((0,T)\times \R^{2d})$}, \quad p < \frac{d+2}{d+1}.
$$

We can now  pass to the limit in the equation
\begin{equation*}
	\begin{split}
		&f^n_t + v \cdot \Grad_x f^n -\Div(f^n \Grad_x \Phi) + \Div_v\left(f^nL[f^n] \right) \\
		&\qquad \qquad =\sigma \Delta_v f^n-\Div_v(f^n(\chi_\lambda(u^n_\delta) - v)).
	\end{split}
\end{equation*}
The only delicate term  is the nonlinear term $f^n\chi_\lambda(u^n_\delta)$ (the other terms are either linear or involve quantities that are more regular than $f^n\chi_{\lambda_n}(u^n)$).
We write
$$ f^n\chi_\lambda(u^n_\delta) =f^n u^n_\delta +  f^n (\chi_\lambda(u^n_\delta)-u^n_\delta).$$
The strong convergence of $u^n_\delta$ implies that the first term converges to $f u_\delta$ in $\mathcal D'$.
And using the fact that $ u^n_\delta \leq \frac{1}{\delta}j^n$, Lemma \ref{cor:entropy} implies that
$u^n$ is bounded in $L^\infty(0,T;L^p(\R^d))$ for $p<\frac{d+2}{d+1}$, and so
$$ \int | u^n_\delta -\chi_{\lambda_n}(  u^n_\delta)| f^n  \, dx\, dv \leq C ||f^n||_{L^\infty}  \int_{  u^n_\delta>{\lambda_n}} | u^n_\delta|  \, dx\, dv $$ 
converges to zero uniformly w.r.t $t$ as $\lambda_n$ goes to infinity. 
We deduce that
$$ f^n\chi_\lambda(u^n_\delta) \to f u_\delta \mbox{ in } \mathcal D'.$$
which concludes the proof.
\end{proof}
%
%

\subsection{Limit as $\delta\to 0$ (and end of the proof of Theorem \ref{thm:main})}
For all $\delta>0$, we have shown that there exists a weak solution $f$ 
to the equation 
\begin{equation}\label{eq:app2}
	\begin{split}
		&f_t + v \cdot \Grad_x f -\Div(f \Grad_x \Phi) + \Div_v\left(fL[f] \right) \\
		&\qquad \qquad\qquad\qquad =\sigma \Delta_v f-\Div_v(f(u_\delta - v)).
	\end{split}
\end{equation}
where $u_\delta = \frac{j}{\delta + \vr}$. 
We now consider a sequence $\delta_n\to 0$ and denote by $f^n=f_{\delta_n}$ the corresponding weak solution of \eqref{eq:app2}.
Proceeding as before, (using the bounds \eqref{eq:Lpn},  Lemma \ref{lem:Gf} and
 the velocity averaging 
Lemma \ref{cor:velocity})
we can show that there exists a function $f$ such that up to a subsequence, the following convergences hold:
\begin{equation}\label{eq:convergence} 
	\begin{split}
			f^n &\weakstar f\quad \text{in $L^\infty((0,T)\times L^1(\R^{2d})\cap L^\infty(\R^{2d}))$},\\
			\vr^n &\rightarrow \vr \quad \text{in $L^p((0,T)\times \R^{2d})$-strong and a.e.} ,\\
			j^n & \rightarrow j \quad \text{in $L^p((0,T)\times \R^{2d})$-storng and a.e.},
	\end{split}
\end{equation}
for any $p < \frac{d+2}{d+1}$.

It remains to show that  $f$ is a weak solution of \eqref{eq:eq}. More precisely, we have to pass to the limit in the following weak formulation of (\ref{eq:app2}):
	\begin{equation}\label{eq:weak2}
		\begin{split}
				&\int_{\R^{2d+1}} - f^n \psi_t - vf^n\Grad_x \psi +f^n\Grad_x \Phi \Grad_x \psi - f^nL[f^n]\Grad_v \psi ~dvdxdt\\
				&+ \int_{\R^{2d+1}}-\sigma  \na_vf^n\na_v \psi - f^n(u_\delta^n-v)\Grad_v \psi~dvdxdt = \int_{\R^{2d}} f^0 \psi(0,\cdot)~dvdx.
		\end{split}
	\end{equation}
where $\psi\in C^\infty_c([0,\infty)\times\R^{2d})$.
Since $f^n$ converges  in the weak-star topology of $L^\infty$, the 
  most delicate term is the term involving $u_\delta^n$ ($u_\delta^n$ is not bounded in any space).
The key lemma is thus the following:
\begin{lemma}\label{lem:convu}
Let $\vphi\in C^\infty_c(\R^d)$ and denote $\vr_\vphi(x,t) = \int_{\R^d}f(x,v,t)\vphi(v)~dv$. Then, (up to another subsequence),  
$$\vr^n_\vphi \rightarrow \vr_\vphi \quad \text{in $L^p((0,T)\times \R^{2d})$ as $n\to \infty$.}
$$
Furthermore,
$$ 
  \rho^n _\vphi u^n_\delta \rightarrow  \rho_\vphi u \quad\mbox{ in $\mathcal D'((0,T)\times\R^d)$ as $n\to \infty$,}
$$
where $u$ is such that $j(x,t)=\rho(x,t) u(x,t)$ a.e.
\end{lemma}

Lemma \ref{lem:convu} implies that for all $\vphi\in C^\infty_c(\R^d)$ and $\phi\in C^\infty_c([0,T)\times\R^d)$ there exists a subsequence such that
\begin{equation*}
	\begin{split}
		&\lim_{n\rightarrow \infty}\int_0^T\int_{\R^{2d}}  f^n (x,v,t) u^n_\delta(x,t)  \vphi(v)\phi(x,t)\, dv\, dxdt \\
		&\qquad = \lim_{n\rightarrow \infty}\int_0^T\int_{\R^{d}}  \rho^n _\vphi (x,t)  u^n_\delta(x,t)  \phi(x,t)\, dxdt\\
		&\qquad =\int_0^T\int_{\R^{d}}   \rho_\vphi(x,t) u(x,t) \phi(x,t)\, dxdt 
		= \int_0^T\int_{\R^{2d}}   f(t,x)u(t,x) \vphi(v)\phi(x,t)\, dxdt.
	\end{split}
\end{equation*} 
We can thus pass to the limit in \eqref{eq:weak2} and show that the function $f$ (which does not depend on the subsequence) satisfies \eqref{eq:weak} for all test function $\psi(x,v,t)=\vphi(v)\phi(x,t)$ in $C^\infty_c([0,T)\times\R^d\times\R^d)$.
We conclude the proof using the density of the sums and products of functions of the form $\vphi(v)\phi(x,t)$ in 
$C^\infty_c([0,T)\times\R^d\times\R^d)$ (recall that $fu$ is in $L^2$).


\begin{proof}[Proof of Lemma \ref{lem:convu}]
For a given  $\vphi\in C^\infty_c(\R^d)$, the same reasoning used to show the convergence of $\rho^n$ implies that, up to a subsequence,
\begin{equation}\label{eq:convrv}
\vr^n_\vphi \rightarrow \vr_\vphi \quad \text{in $L^p((0,T)\times \R^{2d})$ as $n\to \infty$}.
\end{equation}

We now introduce the function
$$ m^n =   \rho^n_\vphi u^n_\delta.$$
We have  the following pointwise bound:
$$ |m^n|\leq ||\vphi||_\infty \int |v | f^n\, dv,$$
and proceeding as in Lemma \ref{lem:jn}, 
we deduce:
$$ ||m^n||_{L^\infty(0,T;L^p(\R^d))} \leq C   \quad \mbox{for all $p<(d+2)/(d+1)$}.$$
Hence, up to a subsequence, 
\begin{equation*}
	m^n \weakstar m  \quad \text{ in $L^\infty(0,T;L^p(\R^d))$},
\end{equation*}
and it only remains to show that
$$
m = \vr_\psi u,\quad \mbox{ where $u$ is such that } j = \vr u.
$$

First, we check that such a function $u$ exists: Consider the set
$$A_R=\{(x,t)\in B_R\times (0,T)\,;\, \rho =0\}.$$
By direct calculation, we see that
\begin{eqnarray*} 
\int_{A_R} |j^n |\, dx\, dt &  \leq & \left( \int_{A_R} \rho^n   |u^n | ^2\, dx\, dt \right)^{1/2} \left( \int_{A_R} \rho^n  \, dx \, dt\right)^{1/2}  \\
&  \leq &C \left( \int_{A_R} \rho^n  \, dx\, dt \right)^{1/2}  \to 0, 
\end{eqnarray*}
and hence $j=0$ a.e. in $A_R$.
Consequently, we can define the function
$$ u(x,t)=\left\{ \begin{array}{ll}\displaystyle \frac{j(x,t)}{\rho(x,t)} & \mbox{ if }  \rho(x,t) \neq 0 \\[8pt] 0 &  \mbox{ if }  \rho(x,t) = 0 \end{array}\right.$$
and we then have $j=\rho u$. It only remains to prove that  $m=\rho_\vphi u$.

First, we observe that a similar argument implies that $m=0$ whenever $\rho_\vphi=0$, so we only have to check that
$$m(x,t)=\rho_\vphi(x,t) u(x,t) \qquad \mbox{ whenever }  \rho_\vphi(x,t)\neq 0.$$
For this purpose, let us consider the set 
$$ B_R^\eps = \{(x,t)\in B_R\times (0,T)\,;\, \rho >\eps\}.$$
Egorov's theorem together with \eqref{eq:convergence} asserts the existence of a set $C_\eta \subset B_R^\eps$ with
$|B_R^\eps\setminus C_R^\eps|\leq \eta$ on which  
 $ \rho^n_\vphi $ and $\rho^n$ converge uniformly on $C_\eta$ to $\rho_\vphi$ and $\rho$. We then have  (for $n$ large enough) 
$$ \rho^n \geq \eps/2 \quad \mbox{ in } C_\eta,$$
and since
$$ m^n = \frac{j^n}{\delta_n+\rho^n} \rho^n_\vphi,$$
we can pass to the limit in $C_\eta$ (pointwise) to deduce
$$ m =  \frac{j}{\rho} \rho_\vphi = u\rho_\vphi\quad  \mbox{ in } C_\eta.$$
Since this holds for all $\eta>0$, we have
$$ m = u\rho_\vphi\quad  \mbox{ in } B_R^\eps,$$
for every $R$ and $\eps$. We conclude that,
$$ m = u\rho_\vphi \mbox{ in } \{\rho>0\}.$$
\end{proof}

\section{Non-symmetric flocking and the Motsch-Tadmor model}\label{sec:MT}
So far we
have limited our attention to the case of symmetric flocking kernel.
In this section, we will extend our existence result to include some non-symmetric kernel $K(x,y)\neq K(y,x)$. 
The critical step is to derive the appropriate energy bound on the solution.
As we will see in Proposition \ref{prop:MT} below, this is rather 
straight forward provided the flocking kernel 
satisfies a condition of the form
\begin{equation*}
	\int_{\R^{d}}K(x,y)\vr(x)~dx \leq C, \quad \forall \vr \in L_+^1(\R^d).
\end{equation*}
In particular, if $K$ is bounded the result follows readily.  
For this reason we will focus 
on the Motsch-Tadmor model for which $K$ depends on $f$ and is singular (see \eqref{eq:Kf}).
To the authors knowledge this is the most difficult case 
currently found in the mathematical literature.
 
Let us recall the Motsch-Tadmor model for flocking:
\begin{equation}\label{eq:MT0}
	f_t + v\cdot \Grad_x f - \Div_v(f\Grad_x \Phi) + \Div_v\left(f\tilde L[f]\right)   = 0,
\end{equation}
with normalized, non-symmetric alignment, given by
\begin{equation}\label{eq:MT1} 
\tilde L[f] = \frac{\int_{\R^d} \int_{R^d} \phi(x-y)  f(y,w) (w-v)\, dw\, dy}{\int_{\R^d} \int_{R^d} \phi(x-y)  f(y,w)\, dw\, dy} = \tilde u -v
\end{equation}
where $\tilde u$ is defined by the equalities
$$ \tilde \rho =  \int_{\R^{2d}}\phi(x-y)  f(y,w)\, dw\, dy, \qquad \tilde \rho \tilde u =  \int_{\R^{2d}}\phi(x-y) w  f(y,w)\, dw\, dy.$$

This model only differs from (\ref{eq:eq}) by the fact that $u$ is replaced by $\tilde u$. Because the function $\tilde u$ involves the convolution with a smooth kernel $\phi$, we expect this function to be smoother than $u$, and the proof of the existence of a solution for (\ref{eq:MT0}) is actually simpler, provided that we can derive the necessary a priori estimates. This is our goal in the remaining parts of this section.

More precisely, we are going  to prove:
\begin{proposition}\label{prop:MT}
Assume that $\phi$ is a smooth nonnegative function and that there exists $r>0$ and $R>0$ such that 
\begin{equation}\label{eq:condphi} 
\phi(x) >0 \mbox{ for } |x|\leq r\, , \qquad  \phi(x)=0  \quad \mbox{ for } |x|\geq R.
\end{equation}
Let $f$ be a smooth solution of (\ref{eq:MT0})-(\ref{eq:MT1}) and define
$$\E(t)= \int_{\R^{2d}} \left(\frac{|v|^2}{2}+\Phi(x)\right) f(x,v,t)\, dx\, dv.
$$
Then, there exists a constant $C\sim  \frac{\sup_{B_R(0)}\phi}{\inf_{B_r(0)}\phi }  \left( \frac{R}{r}\right)^d $ depending only on $\phi$ such that
\begin{equation}\label{eq:MTapriori}
\E(t) \leq \E(0)e^{Ct} .
\end{equation}
\end{proposition}
We note that the constant $C$ in (\ref{eq:MTapriori}) is invariant if we replace $\phi(x)$ by $\phi(\lambda x)$ (in which case both $R$ and $r$ are scaled by a factor $\lambda$) or by $\lambda \phi(x)$ (in which case the $\sup$ and $\inf$ are both scaled by a factor $\lambda$).
In particular, if we take a sequence $\phi_\eps = \eps^{-d} \phi(x/\eps)$, which converges to $\delta_0$ as $\eps\to0$ (so that $\tilde L$ converges with the local alignment term considered in the previous section), the estimate 
(\ref{eq:MTapriori}) holds uniformly with respect to $\eps$.

\begin{proof}[Proof of Proposition \ref{prop:MT}]
Multiplying (\ref{eq:MT0}) by $\frac{|v|^2}{2}+\Phi(x)$ and integrating with respect to $x$ and $v$, we get:
\begin{eqnarray}
\frac{d}{dt} \E(t) & = & \int_{\R^{2d}} f \tilde L[f]\cdot v\, dx\, dv \nonumber\\ 
&= &  \int_{\R^{2d}} f (\tilde u - v )\cdot v\, dx\, dv\nonumber \\ 
&= & - \int_{\R^{2d}} f (\tilde u - v )^2\, dx\, dv + \int_{\R^{2d}} f (\tilde u - v )\cdot \tilde u \, dx\, dv \nonumber\\ 
&\leq & -\frac{1}{2} \int_{\R^{2d}} f (\tilde u - v )^2\, dx\, dv + \int_{\R^{d}} \rho \tilde u^2 \, dx. \label{eq:MT11}
\end{eqnarray}
It remains to see that $\int_{\R^{d}} \rho \tilde u^2 \, dx$ can be controlled by $ \E(t)$.
First, we notice that
$$ \tilde \rho \tilde u^2 \leq \int_{\R^d} \int_{\R^d}\phi(x-y)   w^2 f(y,w,t) \, dw\, dy$$
and so
\begin{equation}\label{eq:almostthere}
\int_{\R^{d}} \rho \tilde u^2 \, dx \leq  \int_{\R^{d}}\int_{\R^{d}}\int_{\R^{d}} \frac{\rho(x,t)}{\tilde \rho(x,t)}\phi(x-y)   w^2 f(y,w,t) \, dw\, dy\, dx  .
\end{equation}
In order to conclude, we thus need the following lemma.
\begin{lemma}\label{lem:MT}
Under the assumptions of Proposition \ref{prop:MT}, there exists a constant $C\sim \frac{\sup\phi}{\inf_{B_r(0)}\phi }  \left( \frac{R}{r}\right)^d $ such that
$$ 
\int_{\R^{d}} \phi(x-y)   \frac{\rho(x)}{\tilde \rho(x)} \, dx \leq C \qquad \forall y\in \R^d
$$
for all nonnegative functions $\rho\in L^1(\R^d)$.
\end{lemma}
Equation (\ref{eq:almostthere}) together Lemma \ref{lem:MT}  now imply
$$
\int_{\R^{d}} \rho \tilde u^2 \, dx \leq C \int_{\R^{2d}}   w^2 f(y,w,t) \, dw\, dy \leq C\E(t)$$
and so (\ref{eq:MT11}) yields
$$ 
\E'(t)\leq C \E(t)$$
which gives the lemma.
\end{proof}

\begin{proof}[Proof of Lemma \ref{lem:MT}]
First, we note that
$$\int_{\R^{d}} \phi(x-y)   \frac{\rho(x)}{\tilde \rho(x)} \, dx \leq  \sup\phi \int_{ B_R(y)}  \frac{\rho(x)}{\tilde \rho(x)} \, dx.
$$
Next, we cover $B_R(y)$ with balls of radius $r/2$ (with $r$ as in (\ref{eq:condphi})): We have 
$$ B_R(y)\subset \bigcup_{i=1}^N B_{r/2}(x_i)$$
with $N\sim(R/r)^d$.
We can thus write
$$\int_{\R^{d}} \phi(x-y)   \frac{\rho(x)}{\tilde \rho(x)} \, dx \leq 
 \sup\phi\sum_{i=1}^N \int_{B_{r/2}(x_i)}  \frac{\rho(x)}{\tilde \rho(x)} \, dx
$$
where
\begin{eqnarray*}
\tilde \rho(x) = \int_{\R^{d}} \phi(x-z)   \rho(z) \, dz\geq \int_{B_{r/2}(x_i)}  \phi(x-z)   \rho(z) \, dz .
\end{eqnarray*}
We deduce
$$\int_{\R^{d}} \phi(x-y)   \frac{\rho(x)}{\tilde \rho(x)} \, dx \leq 
 \sup\phi\sum_{i=1}^N \int_{B_{r/2}(x_i)}  \frac{\rho(x)}{\int_{B_{r/2}(x_i)}  \phi(x-z)   \rho(z) \, dz} \, dx.
$$
and using the fact that when  $ x,z\in B_{r/2}(x_i)$ we have $|x-z|\leq r$, we deduce
\begin{eqnarray*}
\int_{\R^{d}} \phi(x-y)   \frac{\rho(x)}{\tilde \rho(x)} \, dx& \leq & 
 \frac{\sup\phi}{\inf_{B_r(0)}\phi } \sum_{i=1}^N \int_{B_{r/2}(x_i)}  \frac{\rho(x)}{\int_{B_{r/2}(x_i)}   \rho(z) \, dz} \, dx\\
 & \leq &  \frac{\sup\phi}{\inf_{B_r(0)}\phi }  N\\
& \leq &C  \frac{\sup\phi}{\inf_{B_r(0)}\phi }  \left( \frac{R}{r}\right)^d
\end{eqnarray*} 
and the proof is complete.
\end{proof}
\section{Other existence result}\label{sec:other}
In Section \ref{sec:approximate}, we constructed 
a sequence of approximate solutions of our kinetic flocking equation \eqref{eq:eq}. 
In our discussion therein the existence of these approximate solutions
relied on the existence of solutions to equations of the form
\begin{equation*}
	f_t + v\cdot \Grad_x f - \Div_v(f\Grad_x \Phi) + \Div_v\left(L[f]f\right) =\sigma \Delta_v f + \Div_v\left(F vf-Ef\right),
\end{equation*}
where $E$ and $F$ are given functions. 
The purpose of this section is to prove this result.
The precise statement is given in Theorem \ref{thm:twofixed} below. 

We commence by recalling the following result due to Degond \cite{D-1985, D-1986}:
\begin{proposition}\label{prop:degond}
Let $a(v)$ be a bounded function, $\sigma>0$, then for any $E\in [L^\infty(0,T;L^\infty(\R^d))]^d$ and $F\in L^\infty(0,T;L^\infty(\R^d))$, there exists a unique weak solution $f\in C^0(0,T;L^1(\R^d\times\R^d))$ of 
\begin{equation}\label{eq:0}
	f_t + v\cdot \Grad_x f - \Div_v(f\Grad_x \Phi) =\sigma \Delta_v f + \Div_v\left(Fa( v)f-Ef\right).
\end{equation}
Furthermore, $f$ satisfies
\begin{equation}\label{eq:Lp0}
\|f\|_{L^\infty(0,T;L^p(\R^{2d}))} +\sigma \|\Grad_v f^\frac{p}{2}\|_{L^2((0,T)\times \R^d\times \R^d)}^\frac{2}{p} \leq e^{\frac{d \|F\|_{L^\infty} T}{q}}  \|f_0\|_{L^p(\R^{2d})}
\end{equation}
with  $q=\frac{p}{p-1}$, for any $p\in [1,\infty]$,
and
\begin{equation}\label{eq:vm0}
 \int_{\R^{2d}} |v|^2 f(x,v,t)\, dx\, dv \leq C\int_{\R^{2d}} |v|^2 f_0(x,v)\, dx\, dv
\end{equation}
\begin{equation}\label{eq:x20}
 \int_{\R^{2d}}(1+ \Phi(x)) f(x,v,t)\, dx\, dv \leq    C\int_{\R^{2d}}(1+\Phi(x)) f_0(x,v)\, dx\, dv,
\end{equation}
with $C= C(\|\na_x\Phi\|_{L^\infty}, \|E\|_{L^\infty}, \|F\|_{L^\infty},T)$.
\end{proposition}

By passing to the limit (weakly) in \eqref{eq:0}, we deduce the following corollary.
\begin{corollary}\label{cor:degond}
 For any $E\in [L^\infty(0,T;L^\infty(\R^d))]^d$ and  $F\in L^\infty(0,T;L^\infty(\R^d))$ and for any $\sigma\geq 0$, there exists a unique weak solution $f\in C^0(0,T;L^1(\R^d\times\R^d))$ of 
\begin{equation}\label{eq:1}
	f_t + v\cdot \Grad_x f - \Div_v(f\Grad_x \Phi) =\sigma \Delta_v f  + \Div_v\left(F vf-Ef\right)
\end{equation}
Furthermore, $f$ satisfies (\ref{eq:Lp0}), (\ref{eq:vm0}) and (\ref{eq:x20}).
\end{corollary}

The main result of this section is the following:
\begin{theorem}\label{thm:twofixed}
For any $E\in [L^\infty(0,T;L^\infty(\R^d))]^d$ and  $F\in L^\infty(0,T;L^\infty(\R^d))$ and for any $\sigma\geq 0$, there exists a unique weak solution $f\in C^0(0,T;L^1(\R^d\times\R^d))$ of 
\begin{equation}\label{eq:CS}
	f_t + v\cdot \Grad_x f - \Div_v(f\Grad_x \Phi) + \Div_v\left(L[f]f\right) =\sigma \Delta_v f + \Div_v\left(F vf-Ef\right).
\end{equation}
Furthermore, $f$ satisfies (\ref{eq:Lp0}), (\ref{eq:vm0}) and (\ref{eq:x20}).
\end{theorem}
\begin{proof}
	Let us denote by $\tilde E$ and $\tilde F$ the functions $E$ and $F$ in \eqref{eq:CS} 
	(which are given). We will argue the existence of a fixed point in the following sense:
	For $E$ and $F$ given, let $f$ solve \eqref{eq:1} and define the operators 
	\begin{equation}\label{eq:opdef}
	\begin{split}
		T_1(E,F) &= \tilde E + \int_{\R^d}K_0(x,y)f(y,w)w~dwdy, \\
		T_2(E,F) &= \tilde F + \int_{\R^d}K_0(x,y)f(y,w)~dwdy, \\
		T(E,F) &= [T_1(E,F), ~T_2(E,F) ].
	\end{split}
	\end{equation}
	Then, any fixed point $T(E,F) = [E,F]$ is a solution to \eqref{eq:CS}. 
	We will prove the existence of such a fixed point by verifying the postulates 
	of  Schaefer theorem. However, to facilitate this we will work 
	with the space of continuous functions $C^0((0,T)\times \R^d)$ instead 
	of $L^\infty((0,T)\times \R^d)$. We temporarily assume that $\tilde{E}, \tilde{F}$ are in $C^0((0,T)\times \R^d).$ The result follows by passing to the limit. 
	
	Let us first verify that the operator $T$ is compact. For this purpose, let $\{[E^n, F^n]\}_n$
	be a uniformly bounded sequence in $[C^0([0,T)\times \R^d)]^2$ and 
	 $\{f^n\}_n$ be the corresponding sequence solutions to \eqref{eq:1}.
	By virtue of \eqref{eq:vm0} - \eqref{eq:x20}, it is clear 
	that $T(E^n, F^n)$ is uniformly (in $n$) bounded in $[L^\infty((0,T)\times \R^d)]^2$.
	Since $K_0$ is Lipschitz, we have that 
	\begin{equation*}
		\begin{split}
			&\left|\int_{\R^d}K_0(x+\epsilon,y)f^n(y,w)w~dwdy-\int_{\R^d}K_0(x,y)f^n(y,w)w~dwdy \right| 
			 \leq \epsilon C\|K_0\|_{W^{1,\infty}},
		\end{split}
	\end{equation*}
	where $C$ is independent of $n$. In particular, 
	\begin{equation*}
		\left\{\int_{\R^d}K_0(x,y)f^n(y,w)w~dwdy\right\}_n,
	\end{equation*}
	is both uniformly bounded and equicontinuous. Clearly, it follows that $T_1$
	is compact in $C^0((0,T)\times \R^d)$. A similar argument holds for $T_2$ and hence $T$ is compact in $C^0((0,T)\times \R^d)$. 
	
	Next, let us verify that the operator $T$ is continuous. Observe that this actually follows from compactness
    provided that $T(E^n, F^n) \rightarrow T(E, F)$, whenever $[E^n, F^n] \rightarrow [E, F]$
	in $C^0((0,T)\times \R^d)$. In turn, this is immediate if $f^n \weak f$, where 
	$f$ is a weak solution of \eqref{eq:1} with $E$ and $F$ being the above described limits.
	Consequently, we can conclude continuity of $T$ if we can pass to the limit in 
	\begin{equation}\label{eq:2}
		f^n_t + v\cdot \Grad_x f^n - \Div_v(f^n\Grad_x \Phi) =\sigma \Delta_v f^n  + \Div_v\left(F^n vf^n-E^nf^n\right)
	\end{equation}
	Note that the bounds \eqref{eq:Lp0} - \eqref{eq:x20}, together with the assumption that 
	$F^n$ and $E^n$ are uniformly bounded, provides the existence of a constant $C$, independent of
	$n$, 
	\begin{equation*}
	 \sup_{t\in(0,T)}\int_{\R^{2d}} (1+ \Phi(x) +|v|^2) f^n(x,v,t)\, dx\, dv + \|f^n\|_{L^\infty(0,T;L^\infty(\R^d))} \leq C,
	\end{equation*}
	for any given finite total time $T$. 
	Since $F^n$ and $E^n$ converge strongly there is
	no problems with passing to the limit in \eqref{eq:2} to conclude that the limit $f$ solves 
	\eqref{eq:1}.
	Hence, $T(E^n, F^n) \rightarrow T(E, F)$ in $C^0((0,T)\times \R^d)$ and consequently 
	is also continuous.
	
	To conclude the existence of a fixed point, it remains to verify that 
	\begin{equation*}
		\Set{[E,F] = \lambda T(E,F)~  \text{ for some }\lambda \in [0,1]},
	\end{equation*}
	is bounded. For this purpose we take $[E,F]$ in this set and  $f$  a weak solution of
	\begin{equation}\label{eq:CSS}
		f_t + v\cdot \Grad_x f - \Div_v(f\Grad_x \Phi) + \Div_v\left(L[f]f\right) =\sigma \Delta_v f + \Div_v\left(\lambda \tilde F vf-\lambda \tilde Ef\right).
	\end{equation}
	The corresponding energy estimate becomes
	\begin{equation*}
		\begin{split}
			\mathcal{E}(t) &:= \int_{\R^{2d}}(\frac{1}{2}|v|^2 + \Phi)f~dvdx \\
			&\quad + \frac{\lambda}{2}\int_0^t\int_{\R^{4d}}K_0(x,y)f(x,v)f(y,w)|w-v|^2~dwdydvdx \\
			&\leq \lambda \|\tilde F\|_{L^\infty}\int_0^t \mathcal{E}(s)~ds + \lambda \|\tilde E\|_{L^\infty} + \mathcal{E}(0).
		\end{split}
	\end{equation*}
	Since $\lambda \leq 1$, an application of the Gronwall inequality yields
	\begin{equation*}
		\sup_{t\in(0,T)} \mathcal{E}(t) \leq C\left(\|\tilde F\|_{L^\infty},\|\tilde E\|_{L^\infty},T \right)\left(\mathcal{E}(0)+1\right).
	\end{equation*}
	Equipped with this bound, we deduce from \eqref{eq:opdef} that
	\begin{equation*}
		\begin{split}
			\|E\|_{L^\infty} 
			&= \lambda \|T_1(E,F)\|_{L^\infty} \\
			&\leq \|\tilde E\|_{L^\infty} + C\left(\|\tilde F\|_{L^\infty},\|\tilde E\|_{L^\infty} \right)\|K_0\|_{L^\infty}M^\frac{1}{2}\left(\int_{\R^{2d}}|v|^2f_0~dvdx\right) \\
			\|F\|_{L^\infty} &= \lambda \|T_2(E,F)\|_{L^\infty}\leq  \|\tilde F\|_{L^\infty} + \|K_0\|_{L^\infty}M.
		\end{split}
	\end{equation*}

	To summarize, the operator $T$ do satisfy the postulates of the Schaefer fixed point theorem 
	and hence we conclude the existence of a fixed point. This fixed point is a solution 
	of \eqref{eq:CS} and hence our proof is complete.
	
\end{proof}


\section{The entropy of flocking}\label{sec:entropy}

Equations of the form \eqref{eq:eq} are expected 
to possess a natural dissipative structure 
often expressed through the notion of entropy.  
In our existence analysis, we have not relied 
on such inequalities because the energy inequality (\ref{eq:apriori0}) was enough.
This inequality is however of limited use for the analysis of asymptotic behavior involving 
singular noise.
In this section, we prove that our weak solutions 
satisfy bounds akin to classical entropy inequalities when $\beta>0$ and $\sigma>0$  (in the case $\sigma=0$, the computation below reduces to (\ref{eq:apriori0})).
We will restrict to considering \emph{symmetric} flocking 
 as the non-symmetric case does not seem 
to posses any "nice" dissipative structure
(in fact, the non-symmetric case does not even conserve momentum).

We recall that the entropy is given by
$$	\mathcal{F}(f) = \int_{\R^{2d}} \frac{\sigma}{\beta} f\log f + f \frac{v^2}{2} + f\Phi~\, dv\,dx
$$and the associated dissipations by
$$D_1(f)=\frac{1}{2}\int_0^T\int_{\R^{2d}} \frac{\beta}{f}\left|\frac{\sigma}{\beta} \Grad_v f - f(u-v)\right|^2~\, dv\, dx
$$
and 
$$D_2(f)= \frac{1}{2}\int_{\R^d}\int_{\R^d}\int_{\R^d}\int_{\R^d}K_0(x,y)f(x,v)f(y,w)\left|v - w\right|^2~dwdydvdx$$

The first part of Proposition \ref{prop:main} follows from the following lemma:
\begin{lemma}\label{lem:entropyeq}
Let $f$ be a sufficiently integrable solution of \eqref{eq:eq}, then
\begin{align}\label{eq:entropy1'}
		&\partial_t \mathcal{F}(f) + D_1(f) + D_2(f)
	\nonumber\\
		&  =\frac{\sigma}{\beta} d \int_{\R^d}\int_{\R^d}\int_{\R^d}\int_{\R^d}K_0(x,y)f(x,v)f(y,w) ~ dwdydvdx. 
\end{align}
\end{lemma}
\begin{proof}
	Using the equation \eqref{eq:eq}, we calculate
	\begin{equation}\label{eq:st1}
		\begin{split}
				\partial_t \mathcal{F}(f) 
				&= \int_{\R^d}\int_{\R^d}f_t\left(\frac{\sigma}{\beta}\log f + v+ \Phi\right)~dvdx \\
				&= \int_{\R^d}\int_{\R^d}\frac{\sigma}{\beta} L[f]\Grad_v f -\frac{1}{f}\left(\sigma \Grad f -\beta f(u-v)\right)\frac{\sigma}{\beta}\Grad_v f~dvdx\\
				&\quad +\int_{\R^d}\int_{\R^d} fL[f]v  - \left(\sigma \Grad f -\beta f(u-v)v\right)~dvdx\\
				&\quad +\int_{\R^d}\int_{\R^d}-vf\Grad_x \Phi + vf\Grad_x \Phi~dvdx\\
				&= \int_{\R^d}\int_{\R^d} -\frac{\sigma}{\beta} f\Div_v L[f] + vfL[f]~dvdx \\
				&\quad -\int_{\R^d}\int_{\R^d}\frac{\beta}{f}\left(\frac{\sigma}{\beta}\Grad f - f(u-v)\right)(\frac{\sigma}{\beta}\Grad_v f + vf)~dvdx 
				:= I + II.
		\end{split}
	\end{equation}
	By definition of $L[f]$, we deduce
	\begin{equation}\label{eq:I}
		\begin{split}
			I & := \int_{\R^d}\int_{\R^d} -\frac{\sigma}{\beta}f\Div_v L[f] + vfL[f]~dvdx \\
			& = \int_{\R^d}\int_{\R^d}\int_{\R^d}\int_{\R^d} \frac{\sigma}{\beta}K_0(x,y)f(x,v)f(y,w) \Div_v v~ dwdydvdx \\
			&\quad +\int_{\R^d}\int_{\R^d}\int_{\R^d}\int_{\R^d}K_0(x,y)f(x,v)f(y,w)(w - v)v~dwdydvdx \\
			& = \int_{\R^d}\int_{\R^d}\int_{\R^d}\int_{\R^d}\frac{\sigma d}{\beta} K_0(x,y)f(x,v)f(y,w)~ dwdydvdx \\
			&\quad -D_2(f),
		\end{split}
	\end{equation}
	where we have used the symmetry of $K_0(x,y)f(x)f(y)$ to conclude the last equality.

	By adding and subtracting $u$, we rewrite $II$ as follows:
	\begin{align}\label{eq:II}
			II :&= -\int_{\R^d}\int_{\R^d}\frac{\beta}{f}\left(\frac{\sigma}{\beta}\Grad_v f - f(u-v)\right)(\frac{\sigma}{\beta}\Grad_v f + vf)~dvdx \nonumber. \\
			&= -D_1(f) + \int_{\R^d}\int_{\R^d}- u\sigma \Grad_v f + fu(u-v)~dvdx\\
			&= -D_1(f) + \int_{\R^d}\vr u^2 - \vr u^2~dx = -D_1(f).\nonumber
	\end{align}
	We conclude by setting \eqref{eq:I} and \eqref{eq:II} in \eqref{eq:st1}.
\end{proof}

To prove the 
 second inequality (\ref{eq:entropy1.2}) in Proposition \ref{prop:main}, we must prove 
 that the right-hand side in \eqref{eq:entropy1'} can be controlled by 
the dissipation and the entropy. We will need the following classical lemma:
\begin{lemma}\label{lem:confinement}
Let $\vr \in L^1_+(\R^d)$ be a given density and let $\Phi$ be a confinement potential satisfying (\ref{eq:conff}). Then, 
the negative part of $\vr \log_- \vr$ is bounded as follows
\begin{equation}
	\int_{\R^d} \vr \log_-\vr~dx \leq \frac{1}{2}\int_{\R^d}\vr\Phi~dx + \frac{1}{e}\int_{\R^d}e^{-\frac{\Phi}{2}}~dx.
\end{equation}
In particular, we have
\begin{equation}\label{eq:entroconf}
 \int_{\R^d}\int_{\R^d}f\log_+ f + \frac{fv^2}{2} + f\Phi~dvdx \leq C\F(f)
 \end{equation}
\end{lemma}

Inequality (\ref{eq:entropy1.2})  now follows from the following lemma:
\begin{lemma}\label{lem:weirdtermiscontrolled}
Let $g \in L^p(\R^d\times \R^d)$ be given and let $\Phi$ be a confinement potential satisfying (\ref{eq:conff}). There is a constant $C>0$, depending only on $\|K_0\|_\infty$, $\Phi$ and 
the total mass $M$, such that
\begin{equation*}
	\begin{split}
		&\frac{1}{2}\int_{\R^d} \int_{\R^d} K_0(x,y)\vr(x)\vr(y)\left|u(x) - u(y) \right|^2~dydx -\frac{1}{2}D(f) - C\mathcal{F}(f(t))\\
		& \leq D_2(f) - \frac{d\sigma}{\beta}  \int_{\R^d}\int_{\R^d} \int_{\R^d} \int_{\R^d}K_0(x,y)f(x,v)f(y,w)~dwdydvdx.
	\end{split}
\end{equation*}
\end{lemma}
\begin{proof}[Proof of Lemma \ref{lem:weirdtermiscontrolled}]
By symmetry of $K_0(x,y),$ we have
\begin{equation}\label{eq:st}
	\begin{split}
		&\frac{1}{2}\int_{\R^d} \int_{\R^d} K_0(x,y)\vr(x)\vr(y)\left|u(x) - u(y) \right|^2~dydx\\
		&= \int_{\R^d} \int_{\R^d} K_0(x,y)\vr(x)\vr(y)\left(u(x) - u(y) \right)u(x)~dydx\\
		&= \int_{\R^d} \int_{\R^d}\int_{\R^d} \int_{\R^d} K_0(x,y)f(x,v)f(y,w)\left(v - w\right)u(x)~dwdydvdx\\
		&=\int_{\R^d} \int_{\R^d}\int_{\R^d} \int_{\R^d} K_0(x,y)f(y,w)\left(v - w\right) \\ 
		&\qquad \qquad \qquad \qquad \qquad \times \left(f(x,v)(u(x) - v) - \frac{\sigma}{\beta}\Grad_v f(x,v)\right)~dwdydvdx\\
		&\quad +\int_{\R^d} \int_{\R^d}\int_{\R^d} \int_{\R^d} K_0(x,y)f(y,w)f(x,v)\left(v - w\right)v~dwdydvdx \\
		&\quad +\frac{\sigma}{\beta}\int_{\R^d} \int_{\R^d}\int_{\R^d} \int_{\R^d} K_0(x,y)f(y,w)\left(v - w\right)\Grad_v f(x,v)~dwdydvdx \\
		&= I + II + III.
	\end{split}
\end{equation}
Let us first consider the last term. Integration by parts provides the identity
\begin{equation}\label{eq:III2}
	\begin{split}
		III & = \frac{\sigma}{\beta}\int_{\R^d} \int_{\R^d}\int_{\R^d} \int_{\R^d} K_0(x,y)f(y,w)\left(v - w\right)\Grad_v f(x,v)~dwdydvdx\\
		&= -\frac{\sigma d}{\beta}\int_{\R^d} \int_{\R^d}\int_{\R^d} \int_{\R^d}K_0(x,y)f(y,w)f(x,v)~dwdydvdx.
	\end{split}
\end{equation}

By symmetry of the kernel $K_0(x,y)$, we have that
\begin{equation}\label{eq:II2}
	\begin{split}
		II & = \int_{\R^d} \int_{\R^d}\int_{\R^d} \int_{\R^d} K_0(x,y)f(y,w)f(x,v)\left(v - w\right)v~dwdydvdx \\
		& = \int_{\R^d} \int_{\R^d}\int_{\R^d} \int_{\R^d} K_0(x,y)f(y,w)f(x,v)\frac{\left|v - w\right|^2}{2}~dwdydvdx.
	\end{split}
\end{equation}

It remains to bound $I$. For simplicity, let us introduce the notation
$$
V(x,v) = \frac{1}{\sqrt{f(x,v)}}\left(f(x,v)(u(x) - v) - \frac{\sigma}{\beta}\Grad_v f(x,v)\right).
$$
Using this notation, some straight forward manipulations, and the H\" older 
inequality, we obtain using  Lemma \ref{lem:confinement},
\begin{equation}\label{eq:I2}
	\begin{split}
		I&=\int_{\R^d} \int_{\R^d}\int_{\R^d} \int_{\R^d} K_0(x,y)\sqrt{f(x,v)}f(y,w)\left(v - w\right)V(x,v)~dwdydvdx \\ 
		&= \int_{\R^d} \int_{\R^d}\int_{\R^d} K_0(x,y)\sqrt{f(x,v)}\vr(y)(v - u(y))V(x,v)~dydvdx \\
		&= \int_{\R^d} \left(\int_{\R^d} K_0(x,y)\vr(y)~dy\right)\int_{\R^d}v\sqrt{f(x,v)}V(x,v)~dvdx \\
		& \qquad -\int_{\R^d}\left(\int_{\R^d} K_0(x,y)\vr(y)u(y) ~dy\right) \int_{\R^d}\sqrt{f(x,v)}V(x,v)~dvdx \\
		&\leq \|K\|_{L^\infty } M\left(\int_{\R^d}\int_{\R^d}|v|^2f(x,v)~dvdx\right)^\frac{1}{2} \left(\int_{\R^d}\int_{\R^d} |V(x,v)|^2~dvdx\right)^\frac{1}{2} \\ 
		&\quad +\|K\|_{L^\infty } M^\frac{1}{2}\left(\int_{\R^d}f v ~dx\right)\left(\int_{\R^d}\int_{\R^d} |V(x,v)|^2~dvdx\right)^\frac{1}{2} \\ 
		&\leq \frac{C(K,M)}{\beta}\mathcal{F}(f) + \frac{1}{2}D_1(f).
	\end{split}
\end{equation}
We conclude the result 
by setting \eqref{eq:III2} - \eqref{eq:I2} in \eqref{eq:st}.
\end{proof}



\begin{proof}
Let $f$ be the solution of (\ref{eq:app}) given by Proposition \ref{prop:ex}.
A computation similar to the proof of Lemma \ref{lem:entropyeq}
yields
\begin{align}
		&\partial_t \mathcal{F}(f) + D_1(f) + D_2(f)
	\nonumber\\
		&  =\frac{\sigma}{\beta} d \int_{\R^d}\int_{\R^d}\int_{\R^d}\int_{\R^d}K_0(x,y)f(x,v)f(y,w) ~ dwdydvdx + \int f v \left[ \chi_\lambda (u_\delta)-u \right]\,dv dx.\nonumber
\end{align}
where
\begin{align*}
 \int f v \left[ \chi_\lambda (u_\delta)-u \right]\,dv dx
& =   \int \rho u\left[ \chi_\lambda (u_\delta)-u \right] dx\\
& \leq  \frac{1}{2} \int \rho u^2\, dx +\frac{1}{2} \int \rho  \chi_\lambda (u_\delta)^2 \, dx-\int \rho u^2  dx\\
&\leq  0
 \end{align*}
since $|\chi_\lambda (u)|\leq |u|$.
We deduce that the solution of the approximated equation (\ref{prop:ex}) satisfy the entropy inequality
\begin{align*}
		&\partial_t \mathcal{F}(f) + D_1(f) + D_2(f)
	\nonumber\\
		&  \leq \frac{\sigma}{\beta} d \int_{\R^d}\int_{\R^d}\int_{\R^d}\int_{\R^d}K_0(x,y)f(x,v)f(y,w) ~ dwdydvdx.
\end{align*}
Integrating in time and passing to the limit (using the convexity of the entropy), we deduce \eqref{eq:entropyfinal}.

\end{proof}

\end{document}